\documentclass[11pt]{article}
\usepackage{graphicx} % Required for inserting images
\usepackage{amsmath,amsthm}
\usepackage{amssymb}
\usepackage{todonotes}
\usepackage{bbm}
\usepackage{enumitem, hyperref}

\newcommand*{\N}{\mathbb{N}}
\newcommand*{\Z}{\mathbb{Z}}

\newcommand*{\R}{\mathbb{R}}

\newcommand*{\e}{\mathrm{e}}

\newcommand*{\E}{\mathbb{E}}
\newcommand*{\Prob}{\mathbb{P}}
\newcommand*{\maxdis}{\mathsf{Max}}
\newcommand*{\converges}{\xrightarrow[]{n\rightarrow\infty}}
\newcommand*{\CID}{\overset{\mathcal{D}}{\longrightarrow}}
\newcommand*{\CIP}{\overset{\mathbb{P}}{\longrightarrow}}
\newcommand*{\ProbHat}{\widehat{\mathbb{P}}}
\newcommand*{\ExpHat}{\widehat{\mathbb{E}}}
\newcommand*{\cupdot}{\mathbin{\mathaccent\cdot\cup}}

\newtheorem{theorem}{Theorem}
\newtheorem{lemma}{Lemma}

\newtheorem{corollary}{Corollary}
\newtheorem{prop}{Proposition}

\definecolor{caribbeangreen}{rgb}{0.0, 0.8, 0.6}
\usepackage[left=1in, right=1in, bottom = 1in, top = 1in]{geometry}
\title{Limit theorems for the empirical distribution of supercritical branching random walks on transitive graphs}
\author{Robin Kaiser, Martin Klötzer, Ecaterina Sava-Huss}
\date{\today}

\begin{document}
\maketitle

\begin{abstract}
We consider supercritical branching random walks on transitive graphs and we prove a law of large numbers for the mean displacement of the ensemble of particles, and a Stam-type central limit theorem for the empirical distributions, thus answering the questions from Kaimanovich-Woess \cite[Section 6.2]{WoessKaim}. 
\end{abstract}



\section{Introduction}

Limit theorems for random walks - such as law of large numbers and central limit theorems for the rate of escape, the range, and other related quantities - are central questions about random walks and have seen widespread research interest and many contributions in the past, see \cite{SubET,SawSteg}. It is also natural  to ask in what way such results hold for the ensemble of particles in a branching process. Harris \cite{HarrisTOBP} conjectured that a central limit type theorem for the scaled empirical distribution of a supercritical branching process on $\mathbb{Z}$ should hold, and this has been proven in Stam \cite{ConjHarris}. Since then, the results of Stam have been generalized and improved \cite{BigginsCLT,GaoCLT1,GaoCLT2}. See also \cite{LecNotesBRW} and the references therein for an overview on branching random walks.

Besides central limit theorems, supercritical branching process have been studied from many different perspectives, and additive martingales, multiplicative martingales and derivative martingales \cite{AdditiveMartingale,DerivativeMartingale1,DerivativeMartingale2} are core tools in understanding such processes.
For instance, employing such martingales, asymptotics for the maximal drift of supercritical branching processes  can be analyzed \cite{FirstBirthProbBigg,AidekonLeftmost}. For the three aforementioned martingales, Kesten-Stigum convergence type results have been proven in \cite{BigginsMartingale1, BigginsMartingale2}.

In \cite{WoessKaim, Hutch}, supercritical branching random walks on graphs are viewed from the viewpoint of their empirical distributions, and the authors prove the convergence of these distributions to random limit boundary measures which are measures on the boundary of the underlying graph. Recently, in \cite{elisabetta} 
the authors investigate the boundary behaviour of supercritical branching random walks, connecting the typical asymptotic paths of particles to the Martin boundary of the random walk on the underlying graph.
In the current work, motivated by questions raised in \cite[Section 6.2]{WoessKaim}, we analyse the empirical distribution of a supercritical branching process from a quantitative perspective, and we prove a law of large numbers and a Stam type central limit theorem for branching processes on transitive graphs.

Let $G$ be an infinite, connected, transitive graph, and we choose a fixed vertex $o\in G$ to be the origin of $G$. Let $d$ be the graph metric in $G$, and we denote by $|x|=d(o,x)$ the distance of $x$ to the origin. Furthermore, let $\textsf{P}=(\textsf{p}(x,y))_{x,y\in G}$ be the transition kernel of a random walk on $G$. We call the random walk $(Y_n)_{n\in\N}$ with transition probabilities given by $\textsf{P}$ the underlying random walk or the underlying Markov chain. In addition, let $\pi$ be a distribution supported on the natural numbers $\mathbb{N}$, which will serve as the offspring distribution of a branching process. A \textit{branching random walk} $(X_v)_{v\in T}$ over $G$ is a random walk indexed over the Galton-Watson tree $T$ with offspring distribution $\pi$, where the particles move according to the transition kernel $\mathsf{P}$; see Section \ref{sec:prelim} for details.

We write $m_n(x)$ for the random number of particles at location $x\in G$ in generation $n$ and $Z_n$ for the total random number of particles in generation $n$, i.e.~$Z_n = \sum_{x\in G}m_n(x)$. We then define the empirical distribution of the branching random walk as
\begin{equation}\label{eq:emp-distr}
M_n:=\frac{1}{\rho^n}\sum_{x\in G}m_n(x)\delta_x,
\end{equation}
where $\rho$ is the expected offspring number, i.e.~$\mathbb{E}[\pi]=\rho$. Note that unlike in our paper, in \cite{WoessKaim} the empirical distribution is defined as the probability measure 
$Z_n^{-1}\sum_{x\in G} m_n(x)\delta_x$, so our measure differs from the one in \cite{WoessKaim} by the well understood population martingale $Z_n/\rho^n$.
We still call $\rho^{-n}\sum_{x\in G} m_n(x)\delta_x$ the empirical distribution, and this is the main object of interest in the current paper.

We introduce several assumptions on the branching random walk and on the underlying random walk, under which our results hold.
\begin{enumerate}[label=(A\arabic*)]
\setlength\itemsep{-0.35cm}
\item The offspring distribution $\pi$ has finite second moment $\theta$, i.e.
$$\theta=\sum_{k\in\N} k^2 \pi(k) <\infty.$$\label{A1}
\item The branching process is supercritical, i.e.
$$\rho=\sum_{k\in\N} k\pi(k) > 1.$$\label{A2}
\item The underlying random walk with transition matrix $\mathsf{P}$ has some finite exponential moment, i.e.~there exists a $t>0$ such that
$$\sum_{x\in G}\e^{t|x|}\textsf{p}(o,x) < \infty.$$\label{A3}
\item The drift of the underlying random walk is independent of the initial state, i.e.~for random walks $(Y_n^x)_{n\in\N}$ and $(Y_n^y)_{n\in\N}$ started in $x$ and $y$ respectively, it holds that
\begin{equation}\label{eq: indep of drift}
(d(x,Y^x_n))_{n\in\N}\overset{\mathcal{D}}{=}(d(y,Y^y_n))_{n\in\N},
\end{equation}\label{A4}
\item A particle always produces a positive number of offspring, i.e.~$\pi(0)=0$.\label{A5} 
\end{enumerate}
Above ``$\overset{\mathcal{D}}{=}$'' denotes equality in distribution.
We write $W_n$ for the population martingale of the branching process, i.e.~$W_n:=Z_n/ \rho^n$.
Note that condition \ref{A1} implies that 
$$\sum_{k\in\N}k\log(k)\pi(k) <\infty,$$
which is a necessary and sufficient condition for the population martingale $W_n$ to converge almost surely to a limit $W$ with $\Prob(W>0) >0$, in view of Kesten and Stignum \cite{KestenStigum}.
Also note that assumptions \ref{A3} and \ref{A4} together with Kingman's subadditive ergodic theorem \cite{SubET} imply that the underlying random walk $(Y_n)_{n\in\N}$ has a rate of escape, i.e.~the almost sure limit 
$$\ell = \lim_{n\to\infty}\frac{|Y_n|}{n} = \lim_{n\to\infty}\frac{\E\big[|Y_n|\big]}{n}$$ 
exists and is finite.
Our first main result deals with the strong law of large numbers for the empirical distribution $M_n$.
\begin{theorem}\label{thm:first_main}
Under assumptions \ref{A1}-\ref{A5} and $M_n$ defined as in \eqref{eq:emp-distr}, it holds
$$\lim_{n\to\infty}\frac{1}{n}\int |x|\,dM_n(x) = \ell\cdot W\quad \text{almost surely}.$$
\end{theorem}
Note that we have $\frac{1}{n}\int |x|\,dM_n(x) =\frac{1}{n\rho^n}\sum_{x\in G}|x|m_n(x)$.
For the second main result, which is a Stam type central limit theorem for the empirical distributions, let us first define for $a,b\in\mathbb{R}$ with $a\leq b$, the set
$$A(a,b)=\{x\in G:|x|\in[a,b]\}.$$
As usual, $\mathcal{N}(0,1)$ denotes the standard normal distribution and $\Phi$ its distribution function.
\begin{theorem}\label{thm:second_main}
Assume that \ref{A1}-\ref{A5} hold, and assume that for the underlying random walk $(Y_n)_{n\in\N}$ there exists $\sigma >0$, such that the following convergence in distribution holds
$$\frac{|Y_n|-n\ell}{\sqrt{n}\sigma}\xrightarrow[]{\mathcal{D}}\mathcal{N}(0,1), \quad {as }\hspace{0.2cm} n\to\infty.$$
Then for all $x\in\R$ we have
$$M_n\big(A(0,n\ell+x\sqrt{n}\sigma)\big)\xrightarrow[]{\Prob}W\cdot\Phi(x), \quad \text{ as }n \to \infty.$$
\end{theorem}
By Theorem \ref{thm:first_main}, the ensemble of particles after $n$ steps travels on average a distance of $n\ell$. Thus, for two real numbers $a<b$, the quantity $M_n(A(n\ell+a\sqrt{n}\sigma,n\ell+b\sqrt{n}\sigma))$ is a random measure that counts the number of particles in the branching process that deviate at most an order of $\sqrt{n}$ from the average drift $n\ell$, normalized by the average number of particles after $n$ steps. Theorem \ref{thm:second_main} claims that this random number of particles behaves like a normal distribution as $n$ goes to infinity, stretched by the limit $W$ of the population martingale. Thus, the random measure $M_n^*$ defined as
\begin{align}\label{eq:mstar}M_n^*([a,b])=M_n\big(A(n\ell+a\sqrt{n}\sigma,n\ell+b\sqrt{n}\sigma)\big),\end{align}
converges weakly in probability to $W\cdot\mathcal{N}(0,1)$ as $n\to\infty$, where $\mathcal{N}(0,1)$ is a standard normal random variable independent of $W$.

The proofs of the two main results in the setting of transitive graphs
differ from the ones on $\R$ on $\Z$ in \cite{ConjHarris}, in the sense that
none of the quantities of interest can be calculated precisely as in the case of $\Z$. We employ a different strategy, by  carefully splitting the branching process at suitable levels, in order to obtain fitting upper bounds.
The results apply to any supercritical branching random walk on Cayley graphs of finitely generated groups, which satisfies assumptions \ref{A1}-\ref{A5} and for which the rate of escape of the underlying random walk fulfills a strong law of large numbers and a central limit theorem. Examples are Gromov hyperbolic groups  \cite{MR3462675}, free products of finite groups \cite{SawSteg}, or nilpotent groups  \cite{MR1217561}.

Above and for the rest of the paper $"\xrightarrow[]{\mathcal{D}} "$ and $" \xrightarrow[]{\Prob} "$ denote convergence in distribution and in probability, respectively.
We want to emphasize here, that for the proofs of Theorem \ref{thm:first_main} and Theorem \ref{thm:second_main} we do not assume that the random walk $(Y_n)_{n\in \N}$ is transient or recurrent. In the recurrent case, the drift $\ell$ is $0$, and both main theorems still hold.

\textbf{Organization of the paper.} In Section \ref{sec:prelim} we  define branching processes and collect all results and the notation needed. In Section \ref{sec:lln} we first prove a weak version of Theorem \ref{thm:first_main} using the many-to-few formula. We then strengthen the convergence art to almost sure convergence using a decomposition of the branching process. Section \ref{sec:clt} is devoted to proving Theorem \ref{thm:second_main}. While our results hold for general transitive state space, in Section \ref{sec: example} we consider a concrete example, anisotropic random walks on homogeneous trees, and simulate our results. 

\section{Preliminaries}\label{sec:prelim}
This section is devoted to introducing the concepts used throughout our paper. Any assumptions we make on the underlying state space and on the processes under consideration are collected here.

\textbf{Graphs and random walks.}
Let $G=(V,E)$ be a connected infinite graph with vertex set $V$ and edge set $E$. For vertex $x\in V$, we will often use the notation $x\in G$ instead. We also assume that  $G$ is vertex transitive, which means that for any pair of vertices $x,y\in G$, there exists a graph automorphism $\Psi:G\rightarrow G$ with $y=\Psi(x)$. Let $d(\cdot,\cdot)$ be the graph distance in $G$, that is $d(x,y)$ represents the length of a shortest path between $x$ and $y$ in $G$.

Let $\textsf{P}=(\textsf{p}(x,y))_{x,y\in G}$ be a transition kernel on $G$, i.e.~a stochastic matrix indexed over the  vertices of $G$. A random walk on $G$ is  a Markov chain $(Y_n)_{n\in \N}$ with state space $V$, $Y_0=x$ for some $x\in G$, and such that the one-step transition probabilities are given by $\textsf{P}$:
$$\mathbb{P}(Y_{n+1}=y|Y_n=x)=\textsf{p}(x,y),$$
for all $n\in\N$.  We assume that the random walk with transition matrix $\textsf{P}$ has drift independent of its starting position, i.e.~Equation \eqref{eq: indep of drift} holds,
which implies  that one can bound $|Y_n|$ by a random walk on $\Z$. For $\eta_k = d(Y_k,Y_{k+1})$, $k\in \N$, in view of the triangle inequality it holds 
\begin{equation}\label{bound the RW}
|Y_n|\leq \eta_0+\eta_1+\cdots+\eta_{n-1},
\end{equation}
and the increments $\eta_k$ are i.i.d.~due to the Markov property and Equation \eqref{eq: indep of drift}. 

\textbf{Galton-Watson processes and trees.}
Let $\pi$ be a distribution supported on $\N$ with finite expectation $\rho>1$ and finite second moment $\theta$. We call $\pi$ the offspring distribution. Let $(Y_{n,j})_{n,j\in\N}$ be a sequence of i.i.d.~random variables distributed according to $\pi$. A Galton-Watson process $(Z_n)_{n\in\N}$ is a random process with $Z_0=1$ and for $n\geq 1$
$$Z_{n+1}=\sum_{k=1}^{Z_n} Y_{n,k}.$$
The population martingale denoted by $W_n:=Z_n/\rho^n$ is a non-negative martingale, and as such its almost sure limit
$W=\lim_{n\rightarrow\infty}W_n$
exists. By \cite{KestenStigum}, this limit $W$ is almost surely finite and non-negative if the offspring distribution fulfills the $\pi\log \pi$ condition.

A Galton-Watson tree $T$ is a random rooted tree  that is recursively build from the Galton-Watson process $(Z_n)_{n\in\N}$ as follows. We start at level $0$ with a single vertex $r$ called the root. If for some $n\in\N$, the tree has been constructed up to level $n$, with vertices at level $n$ denoted by $v^n_1,...,v^n_{Z_n}$, then the vertex $v^n_i$ for $i\in\{1,...,Z_n\}$ has a random number  $Y_{n,i}$ of children with distribution $\pi$.
We use the notation $T_n$ for $n\in\N$ to denote the vertices in $T$ at level $n$, i.e.~at distance $n$ from the root (note that the quantities $Z_n$ and $|T_n|$ are the same and we use both of them). We denote the level of a vertex $v\in T$ by $|v|$. For vertices $v,w\in T$, we use the notation $v<w$ to mention that $w$ is a predecessor of $v$, that is the unique path from $v$ to $r$ goes through $w$. Notice that every vertex $v\in T_n$ has a unique predecessor $p(v)\in T_{n-1}$ and therefore every vertex $v\in T_n$, has a unique predecessor at level $i$ for all $i\in\{0,...,n-1\}$. One can compute the predecessor at level $i$ by an $n-i$ fold application of $p$ to $v$, i.e.~the unique predecessor of $v$ at level $i$ is $p^{n-i}(v)$.

\textbf{Branching random walk.}
A branching random walk starts with a single particle at time $0$ (level $0$) at a vertex $o\in G$, which we will call the origin. We write $|x| = d(o,x)$ for $x\in G$. 
At every time $n$, each particle at level $n$ splits into particles whose number is distributed according to $\pi$ and then it dies, and all offspring particles take a random step distributed according to the transition kernel $\textsf{P}$ and independently from each other. Formally, this can be described as a random walk $(X_v)_{v\in T}$ indexed by a Galton-Watson tree $T$ and defined inductively as follows. Let the branching walk start with $X_r=o$. Then for all $v\in T$, given the particle location $X_{p(v)}=x$ for $x\in G$, $X_v$ is then distributed according to the distribution $(\textsf{p}(x,y))_{y\in G}$ independently of all previous actions performed.
The empirical distribution of the branching process $(X_v)_{v\in T}$ is defined as the random measure
$$M_n:=\frac{1}{\rho^n}\sum_{x\in G}m_n(x)\delta_x,$$
where $m_n(x)$ denotes the random number of particles at location $x\in V$ in the $n$-th generation (after $n$ steps of the process):
$$m_n(x)=|\{v\in T_n|X_v=x\}|.$$
By rewriting the empirical distribution in terms of the tree indexed random walk, we have
$$M_n = \frac{1}{\rho^n}\sum_{v\in T_n}\delta_{X_v}.$$
See \cite{BPIntro} for an introduction to branching processes.

\textbf{Many-to-few formula.} We will use several times the so-called many-to-few formula, which we briefly describe below. This allows to relate the moments of the branching random walk with offspring distribution $\pi$ to the underlying random walk $(Y_n)$ with transition kernel $\mathsf{P}$. We state here this principle only in two special cases needed for our purposes; for a more general statement of the many-to-few lemma we refer the reader to \cite{ManyToFew}. We start with the many-to-one lemma which needs no special construction and can be easily proven.
\begin{lemma}[Many-to-one \cite{ManyToFew}]\label{lem: many-to-one}
Let $f:G\to\R$ be a measurable function. Then for any $n\in\N$ it holds
$$\mathbb{E}\Big[\sum_{v\in T_n}f(X_v)\Big] = \rho^n\E[f(Y_n)],$$
where $Y_n$ is a random walk on $G$ with transition kernel $\mathsf{P}$. 
\end{lemma}
\begin{proof}
By conditioning on $T_n$ and using the linearity of the expectation we obtain
\begin{align*}
\mathbb{E}\Big[\sum_{v\in T_n}f(X_v)\big|T_n\Big] = \sum_{v\in T_n}\E[f(X_v)|T_n] = Z_n\E[f(Y_n)]. 
\end{align*}
Taking the expectation over the equation above yields
$$\mathbb{E}\Big[\sum_{v\in T_n}f(X_v)\Big] = \E[Z_n]\E[f(Y_n)] = \rho^n\E[f(Y_n)],$$
and this proves the claim.\end{proof}
While the above result reduces the behaviour of the branching random walk to the behaviour of a single random walk, the question of estimating second moments becomes one about two dependent random walks, a \textit{many-to-two formula}. For stating it, we first construct a modified branching process. 
Let us first - for later ease of notation - redefine the first and second moment of $\pi$ as
\begin{align*}
    w_1=\rho\quad \text{and}\quad w_2=\theta.
\end{align*}
We define a new probability space with new probability measure $\ProbHat$, which has  distinguished lines of descent which we call spines; the corresponding expectation is denoted by $\ExpHat$. Shortly $\ProbHat$ is simply an extension of $\mathbb{P}$ in that all particles behave as in the original branching process with the exception that some particles carry marks showing that they are part of a spine. These marks start at the root and move from one level of the Galton-Watson process to the next, by picking one of their descendants uniformly at random an independently of each other. Of key interest is the time at which the marks split, as up to that time the marks follow a common path, whereas after their splitting time, they  move independently of each other. Under $\ProbHat$ particles behave as follows:
\begin{itemize}
\setlength\itemsep{0em}
    \item We start with one particle at the origin $o$ which (as well as its position) carries two marks $1,2$.
    \item We think of each of the marks as a spine, and write $\xi_n^i$ for whichever particle carries mark $i$  at time $n$, and  $\zeta_n^i = X_{\xi_n^i}$ for  its position in $G$, for $i=1,2$.
    \item Particles who have no marks move according to the underlying chain and branch according to the offspring distribution $\pi$.
    \item Particles who have $j\in\lbrace 1,2 \rbrace$ marks move according to the underlying chain but branch at a modified rate
$$\pi^{(j)}(k) = \pi(k)\frac{k^j}{w_j}. $$
    \item For every $n\in\N$ the mark $\xi_{n+1}^i$ is chosen uniformly among all descendants of $\xi_n^i$. Therefore, for $i\in\lbrace 1,2\rbrace$ the chain $(\zeta_n^i)_{n\in\N}$ moves according to the underlying random walk, and if the spine particles have split, they move independent of each other.  
\end{itemize}
Let us now define the skeleton up to time $n$ as
$$\text{skel}(n):=\{\xi^i_j|i\in\{1,2\},1\leq j\leq n\}\setminus \lbrace r\rbrace,$$
which represents the collection of particles that have carried at least one spine up to time $n$. Particles not in the skeleton (those carrying no marks) have children according to distribution $\pi$.
For $v\in T$, we write $D_v$ for the number of spinal particles passing through $v$:
$$D_v:=|\{i\in\{1,2\}|\xi_{|v|}^i=v\}|.$$
With this extended probability space, we can now state the many-to-two formula. For a more general version and a proof see \cite[Lemma 8]{ManyToFew}. 
\begin{lemma}[Many-to-two \cite{ManyToFew}]\label{lem:many-to-two}
Let $f:G^2\rightarrow \R$ be a measurable function. Then for any $n\in \N$ it holds
\begin{align*}
    \mathbb{E}\Big[\sum_{v_1,v_2\in T_n}f(X_{v_1},X_{v_2})\Big]=\ExpHat\Big[f(\zeta_n^1,\zeta_n^2)\prod_{v\in\text{skel}(n)}w_{D_{p(v)}}\Big],
\end{align*}
where $p(v)$ denotes the unique immediate predecessor of $v$ in $T$.
\end{lemma}
Let us also define the first time the two spine particles split as 
$$\tau:=\inf\{i\in\N:\xi^1_i\neq\xi^2_i\}-1.$$
Throughout this paper we will condition several times on $\tau$, because on the event that $\tau$ is known, the random product in Lemma \ref{lem:many-to-two} reduces to an easy deterministic expression, that we state in the next result, and whose proof is visualized in Figure \ref{fig:many-two}.

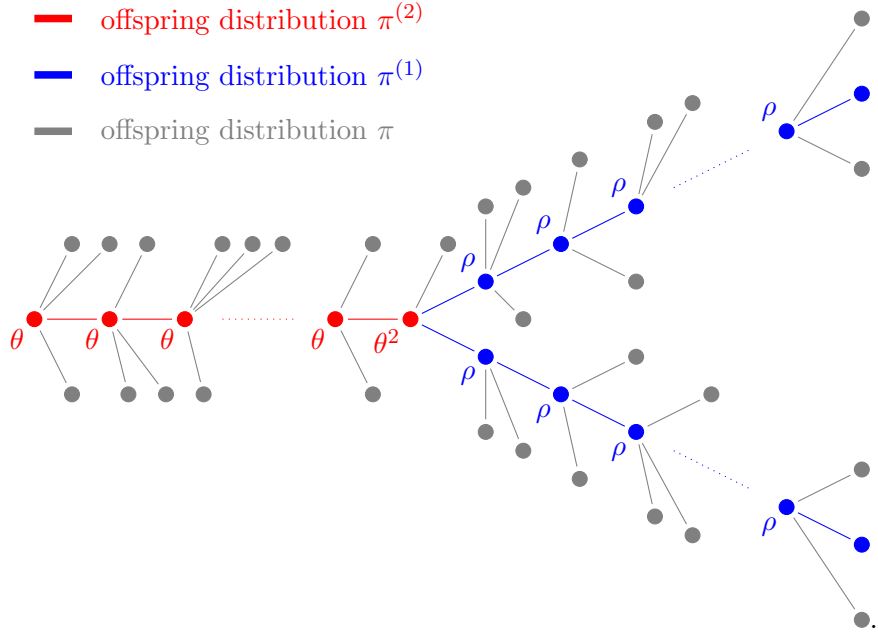
\begin{figure}[h]\label{fig:many-two}
\begin{center}
\begin{tikzpicture}[
  main/.style={circle, color = gray , draw, fill = gray, inner sep = 0.07cm, outer sep = 0.07cm},
  edge/.style={draw, color = gray}
]

\node[main, color = red] (A1) at (0,0) {};
\node[main, color = red] (A2) at (1,0) {};
\node[main, color = red] (A3) at (2,0) {};
\node[main, color = red] (A4) at (4,0) {};
\node[main, color = red] (A5) at (5,0) {};

\node[main, color = blue] (B1) at (6,-0.5) {};
\node[main, color = blue] (B2) at (7,-1) {};
\node[main, color = blue] (B3) at (8,-1.5) {};
\node[main, color = blue] (B4) at (10,-2.5) {};
\node[main, color = blue] (B5) at (11,-3) {};

\node[main, color = blue] (C1) at (6,0.5) {};
\node[main, color = blue] (C2) at (7,1) {};
\node[main, color = blue] (C3) at (8,1.5) {};
\node[main, color = blue] (C4) at (10,2.5) {};
\node[main, color = blue] (C5) at (11,3) {};

\node[main] (D1) at (0.5,1) {};
\node[main] (D1a) at (1,1) {};
\node[main] (D2) at (1.5,1) {};
\node[main] (D3) at (2.5,1) {};
\node[main] (D3a) at (2.9,1) {};
\node[main] (D3b) at (3.3,1) {};
\node[main] (D4) at (4.5,1) {};
\node[main] (D5) at (5.5,1) {};

\node[main] (E1) at (0.5,-1) {};
\node[main] (E2) at (1.25,-1) {};
\node[main] (E2a) at (1.75,-1) {};
\node[main] (E3) at (2.25,-1) {};
\node[main] (E4) at (4.5,-1) {};

\node[main] (F1) at (6.5,-1.75) {};
\node[main] (F1a) at (6,-1.5) {};
\node[main] (F2) at (7.25,-2.125) {};
\node[main] (F2a) at (8,-0.5) {};
\node[main] (F3) at (8.25,-2.625) {};
\node[main] (F3a) at (8.75,-2.875) {};
\node[main] (F3b) at (9,-1) {};
\node[main] (F4) at (11,-4) {};
\node[main] (F4a) at (11,-2) {};

\node[main] (G1) at (6.5,1.75) {};
\node[main] (G1a) at (6,1.5) {};
\node[main] (G1b) at (6.5,0) {};
\node[main] (G2) at (7.25,2.125) {};
\node[main] (G2a) at (8,0.5) {};
\node[main] (G3) at (8.25,2.625) {};
\node[main] (G3a) at (8.75,2.875) {};
\node[main] (G4) at (11,4) {};
\node[main] (G4a) at (11,2) {};

\draw[dotted, color = red] (2.5,0) -- (3.5,0);
\draw[dotted, color = blue] (8.5,1.75) -- (9.5,2.25);
\draw[dotted, color = blue] (8.5,-1.75) -- (9.5,-2.25);

\node[below left, color = red] at (0,0) {$\theta$};
\node[below left, color = red] at (1,0) {$\theta$};
\node[below left, color = red] at (2,0) {$\theta$};
\node[below left, color = red] at (4,0) {$\theta$};
\node[below left, color = red] at (5,0) {$\theta^2$};

\node[below left, color = blue] at (6,-0.5) {$\rho$};
\node[below left, color = blue] at (7,-1) {$\rho$};
\node[below left, color = blue] at (8,-1.5) {$\rho$};
\node[below left, color = blue] at (10,-2.5) {$\rho$};

\node[above left, color = blue] at (6,0.5) {$\rho$};
\node[above left, color = blue] at (7,1) {$\rho$};
\node[above left, color = blue] at (8,1.5) {$\rho$};
\node[above left, color = blue] at (10,2.5) {$\rho$};

\draw[edge, color = blue] (A5) -- (B1);
\draw[edge, color = blue] (A5) -- (C1);

\draw[edge, color = red] (A1) -- (A2);
\draw[edge, color = red] (A2) -- (A3);
\draw[edge, color = red] (A4) -- (A5);

\draw[edge, color = blue] (B1) -- (B2);
\draw[edge, color = blue] (B2) -- (B3);
\draw[edge, color = blue] (B4) -- (B5);

\draw[edge, color = blue] (C1) -- (C2);
\draw[edge, color = blue] (C2) -- (C3);
\draw[edge, color = blue] (C4) -- (C5);

\draw[edge] (A1) -- (D1);
\draw[edge] (A1) -- (D1a);
\draw[edge] (A2) -- (D2);
\draw[edge] (A3) -- (D3);
\draw[edge] (A3) -- (D3a);
\draw[edge] (A3) -- (D3b);
\draw[edge] (A4) -- (D4);
\draw[edge] (A5) -- (D5);

\draw[edge] (A1) -- (E1);
\draw[edge] (A2) -- (E2);
\draw[edge] (A2) -- (E2a);
\draw[edge] (A3) -- (E3);
\draw[edge] (A4) -- (E4);

\draw[edge] (B1) -- (F1);
\draw[edge] (B1) -- (F1a);
\draw[edge] (B2) -- (F2);
\draw[edge] (B2) -- (F2a);
\draw[edge] (B3) -- (F3);
\draw[edge] (B3) -- (F3a);
\draw[edge] (B3) -- (F3b);
\draw[edge] (B4) -- (F4);
\draw[edge] (B4) -- (F4a);

\draw[edge] (C1) -- (G1);
\draw[edge] (C1) -- (G1a);
\draw[edge] (C1) -- (G1b);
\draw[edge] (C2) -- (G2);
\draw[edge] (C2) -- (G2a);
\draw[edge] (C3) -- (G3);
\draw[edge] (C3) -- (G3a);
\draw[edge] (C4) -- (G4);
\draw[edge] (C4) -- (G4a);

\draw[edge, line width = 0.1cm, color = red] (0,4) -- (0.5,4);
\draw[edge, line width = 0.1cm, color = blue] (0,3.25) --(0.5,3.25);
\draw[edge, line width = 0.1cm, color = gray] (0,2.5) -- (0.5,2.5);

\node[right, color = red] at (0.75,4) {offspring distribution $\pi^{(2)}$};
\node[right, color = blue] at (0.75,3.25) {offspring distribution $\pi^{(1)}$};
\node[right, color = gray] at (0.75,2.5) {offspring distribution $\pi$};
\end{tikzpicture}.
\end{center}
\caption{The following diagram visualizes the proof of Lemma \ref{lem: Lemma many-to-two product}: The red and blue coloured part is the skeleton, on red vertices the offspring distribution is $\pi^{(2)}$, on blue vertices the  offspring distribution is  $\pi^{(1)}$, and on grey vertices the offspring distribution is $\pi$. For every vertex $v$ on the skeleton we look at the parent $p(v)$ and add either the first moment $\rho$ or the second moment $\theta$ to $p(v)$, depending on whether $p(v)$ is on the blue or the red part (whether there are one or two spine particles at $p(v)$). When taking the product, on the event $\lbrace \tau = k \rbrace$, the red part gives a contribution of $\theta^{k}\theta^2$ and the blue part a contribution of $\rho^{2(n-k-1)}$, which yields the desired result.}
\end{figure}
\begin{lemma}\label{lem: Lemma many-to-two product}
On the event $\lbrace \tau \geq n \rbrace$ it holds that
$$\prod_{v\in\text{skel}(n)}w_{D_{p(v)}} = \theta^n,$$
and on the event $\lbrace \tau = k \rbrace$ for some $k<n$ it holds that
$$\prod_{v\in\text{skel}(n)}w_{D_{p(v)}} = \rho^{2n}\left(\frac{\theta}{\rho}\right)^2\left(\frac{\theta}{\rho^2}\right)^k.$$
\end{lemma}
\begin{proof}
If $\tau\geq n$ then the spine particles have not split up to time $n$, and therefore the skeleton is just the path of the first (and second) spine particle, i.e.
$$\text{skel}(n) = \lbrace \xi^1_1,\xi_2^1,\dots,\xi^1_n\rbrace = \lbrace \xi^2_1,\xi_2^2,\dots,\xi^2_n\rbrace.$$  
Since in this case on every $v\in\text{skel}(n)$ there are two spine (marked) particles, we have $w_{D_{p(v)}}=\theta$ for every $v\in\text{skel}(n)$, which implies that
$$\prod_{v\in\text{skel}(n)}w_{D_{p(v)}} = \theta^n.$$
Now let $k<n$. Then on the event that the spine particles split at some time $k$, i.e.~on the event $\lbrace \tau = k \rbrace$, the skeleton can be written as
\begin{align*}
\text{skel}(n) &= \lbrace \xi^1_1,\xi^1_2\dots\xi^1_k\rbrace \cupdot \lbrace \xi^1_{k+1},\xi^1_{k+2},\dots,\xi^1_n \rbrace \cupdot
\lbrace \xi^2_{k+1},\xi^2_{k+2},\dots,\xi^2_n \rbrace\\
&=\lbrace \xi^2_1,\xi^2_2\dots\xi^2_k\rbrace \cupdot \lbrace \xi^1_{k+1},\xi^1_{k+2},\dots,\xi^1_n \rbrace \cupdot
\lbrace \xi^2_{k+1},\xi^2_{k+2},\dots,\xi^2_n \rbrace.
\end{align*}
Here, $\cupdot$ denotes the disjoint union of sets. For every $v\in  \lbrace \xi^1_1,\xi^1_2\dots\xi^1_k\rbrace$  it clearly holds that $w_{D_{p(v)}} = \theta$, since for every $v\in  \lbrace \xi^1_1,\xi^1_2\dots\xi^1_k\rbrace$ it holds that $p(v)\in  \lbrace \xi^1_1,\xi^1_2\dots\xi^1_k\rbrace$ and $D_{p(v)} = 2$. Thus the total contribution of this set to the product is $\theta^k$. For both $\xi^1_{k+1}$ and $\xi^2_{k+1}$ their parent is given by $\xi_k^1 = \xi_k^2$, and therefore for both factors we get a contribution of $\theta$, and so in total a contribution of $\theta^2$. Moreover we get a contribution of $\rho$ for all elements in $\lbrace \xi^1_{k+2},\xi^1_{k+3},\dots,\xi^1_n \rbrace \cupdot
\lbrace \xi^2_{k+2},\xi^2_{k+3},\dots,\xi^2_n \rbrace$ and thus a total contribution of $\rho^{2(n-k-1)}$. This yields
$$\prod_{v\in\text{skel}(n)}w_{D_{p(v)}} = \theta^k\theta^2\rho^{2(n-k-1)} = \rho^{2n}\left(\frac{\theta}{\rho}\right)^2\left(\frac{\theta}{\rho^2}\right)^k$$ 
on the event $\lbrace \tau = k \rbrace$.
\end{proof}

\paragraph{Characteristic functions and Levy's continuity theorem.} For the proof of Theorem \ref{thm:second_main}, we use  characteristic functions and a generalized version of Levy's continuity theorem, that we state below. For a real valued random variable $X$ its characteristic function is defined as 
$$\varphi_X(t) =  \E[\e^{itX}], \quad t\in\R,$$
and for a finite measure (not necessarily a probability measure) $\mu$ on $\R$ its characteristic function is defined as 
$$\varphi_\mu(t) = \int \e^{itx} \,d\mu(x), \quad t\in\R.$$
Recall that a sequence of measures $(\mu_n)_{n\in\N}$ on $\R$ is called tight if for every $\varepsilon > 0$ there exists a $K>0$ such that $\sup_{n\in\N} \mu_n([-K,K]^c) < \varepsilon$.
The classical Levy's continuity theorem states that under the condition that the sequence  $(\mu_n)_{n\in\N}$ is uniformly bounded, pointwise convergence of its characteristic functions together with tightness implies weak convergence.  In our case, in the proof of Theorem \ref{thm:second_main}, the appearing sequence of characteristic functions does not necessarily converge pointwise but rather on averages over intervals. Therefore, a weaker version of Levy's continuity theorem is needed. Recall that $(\mu_n)_{n\in\N}$ converges weakly to a measure $\mu$ if
\begin{align}\label{def: weak_conv}
\lim_{n\to\infty} \int f \, d\mu_n = \int f \, d\mu, \quad \text{for all } f\in\mathcal{C}_b(\R),
\end{align}
where $\mathcal{C}_b(\R)$ represents the set of bounded continuous real valued functions on $\R$. The idea behind Levy's continuity theorem is that the family of functions $\mathcal{S} = \lbrace x\mapsto\e^{itx}: t\in\R\rbrace$ is big enough that the following convergence
\begin{align*}
   \lim_{n\to\infty} \int f \, d\mu_n = \int f \, d\mu, \quad \text{for all } f\in\mathcal{S},
\end{align*}
(pointwise convergence of the characteristic functions) implies the convergence in Equation \eqref{def: weak_conv}. The decisive property $\mathcal{S}$ has is that it is separating.
We call the subset $\mathcal{S}\subset\mathcal{C}_b(\R)$ separating if for all finite measures $\mu\neq\nu$ on $\R$ there exists a function $f\in\mathcal{S}$ such that 
\begin{align*}\int f\,d\mu \neq \int f\,d\nu.\end{align*}
The classical Levy' s continuity theorem can now be generalized to any separating family $\mathcal{S}\subset \mathcal{C}_b(\R)$. We will use this result in order to prove the central limit theorem for the empirical distributions.
\begin{theorem}\cite[Theorem 13.34]{Klenke} \label{thm: Levy}
Let $(\mu_n)_{n\in\N}$ be a sequence of measures on $\R$  such that  $\sup_{n\in\N}\mu_n(\R)<\infty$ and $\mu$ be a finite measure on $\R$. Then the following statements are equivalent:
\begin{itemize}
\setlength\itemsep{0cm}
\item[$(i)$]  $\lim_{n\to\infty}\mu_n=  \mu$ weakly.
\item[$(ii)$] $(\mu_n)_{n\in\N}$ is tight, and there exists a separating family $\mathcal{S}\subset \mathcal{C}_b(\R)$ such that for every $f\in \mathcal{S}$
\begin{align*}
  \lim_{n\to\infty}  \int  f\, d \mu_n = \int f \, d\mu.
\end{align*}
\end{itemize}
\end{theorem}

\section{Law of large numbers for the empirical distribution}\label{sec:lln}
In this section we prove Theorem \ref{thm:first_main}. We first show a weaker version of Theorem \ref{thm:first_main}, where instead of almost sure convergence, we prove the $L^2$ convergence. We then use a decomposition of the branching process to lift the $L^2$ convergence to almost sure convergence.

\subsection{Weak law of large numbers}
The proof of the weak law of large numbers for the empirical distribution relies on the previously introduced many-to-two formula in Lemma \ref{lem:many-to-two}. Recall that by $(Y_n)_{n\in\N}$ we denote a random walk with transition probabilities given by $\mathsf{P}$, as defined in Section \ref{sec:prelim}. 
Recall the probability measure $\ProbHat$, where we added two independent spinal particles $(\xi^1_n)_{n\in\N}$ and $(\xi^2_n)_{n\in\N}$ to the random Galton-Watson tree $T$ with offspring distribution $\pi$. By $\tau$ we have denoted the first time the two spine particles split, and we call $\tau$ \textit{the splitting time}.
We introduce the average distance $H_n$ traveled by the particles up to time $n$ as
$$H_n=\frac{1}{\rho^n}\sum_{v\in T_n}|X_v|,$$
and we write
$$L_n=\frac{H_n}{n}=\frac{1}{n\rho^n}\sum_{x\in G}|x|m_n(x).$$
Note that it holds that
$$H_n=\int |x|\,dM_n(x).$$
The main goal in this section is to prove the following proposition, which claims that $L_n$ converges in $L^2$-distance to $\ell W$.
\begin{prop}\label{prop:weak-lln}
Under assumptions \ref{A1}-\ref{A5}, we have
\begin{align*}
\lim_{n\to\infty}\E[(L_n-\ell W_n)^2] = 0,
\end{align*}
where $\ell\geq 0$ is the drift of the underlying random walk, i.e.~the almost sure limit
\begin{align*}
\ell = \lim_{n\to\infty}\frac{|Y_n|}{n}.
\end{align*}
\end{prop}
Since $L^2$ convergence implies convergence in probability, and the variance of $W_n$ is uniformly bounded, we then obtain from Proposition \ref{prop:weak-lln} convergence in probability of $L_n$ to $\ell\cdot W$.
\begin{corollary}
The random variable $L_n$ converges in probability to the random variable $\ell\cdot W$, i.e.~for every $\varepsilon>0$, it holds 
$$\lim_{n\to\infty}\Prob(|L_n-\ell W|>\varepsilon) = 0.$$  
\end{corollary}
For the proof of Proposition \ref{prop:weak-lln}, we need to estimate the second moment of $L_n$, for which we will use the many-to-two formula. The first step in our proof is to analyse the decay of the splitting time of the spinal particles.
\begin{lemma}\label{lemma: decay of splitting time}
We have
$$\Big(\frac{\theta}{\rho}\Big)^2\sum_{k=0}^{\infty}\ProbHat(\tau=k) \Big(\frac{\theta}{\rho^2}\Big)^k=1+\frac{\theta-\rho^2}{\rho(\rho-1)} = \E[W^2].$$
\end{lemma}
\begin{proof}
 From the equation
\begin{align*}
\mathbb{E}\big[Z_{n+1}^2\big\vert Z_n\big]&=\E\bigg[\bigg(\sum_{k=1}^{Z_n} Y_{n,k}\bigg)^2 \Big|Z_n\bigg]= \sum_{k=1}^{Z_n}\sum_{j=1}^{Z_n}\E\big[Y_{n,k}Y_{n,j}\big]\\
&=\sum_{\substack{j,k=1\\j\neq k}}^{Z_n} \E\big[Y_{n,k}\big]\E\big[Y_{n,j}\big] + \sum_{l=1}^{Z_n}\E\big[Y_{n,l}^2\big]
\\&=(Z_n^2-Z_n)\rho^2+Z_n\theta,
\end{align*}    
since $W_n=Z_n/\rho^n$, it follows that
    $$\mathbb{E}\big[W_{n+1}^2\big]=\mathbb{E}\big[W_n^2\big]+\frac{1}{\rho^{n}}\Big(\frac{\theta-\rho^2}{\rho^2}\Big).$$
This recursive equation can be solved and yields the closed-term expression
    $$\mathbb{E}\big[W_n^2\big]=1+\Big(\frac{\theta-\rho^2}{\rho^2}\Big)\sum_{k=0}^{n-1}\frac{1}{\rho^{k}}.$$
The right-hand side above converges to 
$$1 + \Big(\frac{\theta-\rho^2}{\rho^2}\Big)\frac{\rho}{\rho-1} = 1 + \frac{\theta-\rho^2}{\rho(\rho-1)}.$$    
We can now calculate the second moment of the population martingale $W_n$ by using the many-to-two formula (Lemma \ref{lem:many-to-two}) together with Lemma \ref{lem: Lemma many-to-two product}, and we obtain
    \begin{align*}
        \mathbb{E}\big[W_n^2\big]&=
        \frac{1}{\rho^{2n}}\mathbb{E}\big[Z_n^2\big]=
        \frac{1}{\rho^{2n}}\mathbb{E}\Big[\sum_{v,w\in T_n}1\Big]\\
        &=\frac{1}{\rho^{2n}}\ExpHat\Big[\prod_{v\in\text{skel}(n)}w_{D_{p(v)}}\Big]\\
        &=\ProbHat(\tau\geq n)\Big(\frac{\theta}{\rho^2}\Big)^n
        + \Big(\frac{\theta}{\rho}\Big)^2\sum_{k=0}^{n-1}\ProbHat(\tau=k)\Big(\frac{\theta}{\rho^2}\Big)^k.
    \end{align*}
    Since $\mathbb{E}\big[W_n^2\big]$ converges as $n\rightarrow\infty$, so must the right-hand side of the equation above. Thus it holds
    $$\sum_{k=0}^{\infty}\ProbHat(\tau=k) \Big(\frac{\theta}{\rho^2}\Big)^k<\infty.$$
Since the tail end of a convergent series converges to $0$, we have 
    $$\ProbHat(\tau\geq n)\Big(\frac{\theta}{\rho^2}\Big)^n\leq\sum_{k=n}^{\infty}\ProbHat(\tau=k)\Big(\frac{\theta}{\rho^2}\Big)^k\xrightarrow[]{n\rightarrow\infty}0,$$
from which the claim follows.
\end{proof}

In the proof of Proposition \ref{prop:weak-lln} we also need an auxiliary result concerning the expected maximal displacement of the underlying random walk $(Y_n)_{n\in \N}$, which we prove next.
\begin{lemma}\label{lem: WLLN helping lemma 1}
Under assumptions \ref{A3}-\ref{A4}, it holds
$$\lim_{n\to\infty} \max_{k\leq \sqrt{n}} \E\left[\frac{|Y_{n-k}|}{n}\right] = \ell.$$
\end{lemma}
\begin{proof}
First note that $\E[|Y_n|/n]$ converges to $\ell$ due to Kingman's subadditive ergodic theorem. We set $\beta_n$ to be the natural number smaller than $\sqrt{n}$ such that 
$$\max_{k\leq \sqrt{n}} \E\left[\frac{|Y_{n-k}|}{n}\right] =  \E\left[\frac{|Y_{n-{\beta_n}}|}{n}\right].$$
Since $\lim_{n\to\infty}n-\beta_n=\infty$ and 
$\lim_{n\to\infty}\frac{\beta_n}{n} = 0$, we have 
$\lim_{n\to\infty}\frac{n-\beta_n}{n} = 1$,
which implies
$$
\E\left[\frac{|Y_{n-{\beta_n}}|}{n}\right] = \frac{n-\beta_n}{n}\E\left[\frac{|Y_{n-{\beta_n}}|}{n-\beta_n}\right]\xrightarrow[]{n\rightarrow\infty} \ell.
$$
\end{proof}
We are now in position to prove Proposition \ref{prop:weak-lln}.
\begin{proof}[Proof of Proposition \ref{prop:weak-lln}]
We show that $L_n$ stays close to $\ell\cdot W_n$ in mean-square distance. We have
$$\mathbb{E}[(L_n-\ell W_n)^2]=\E[L_n^2]+\ell^2\E[W_n^2]-2\ell\E[L_nW_n].$$
Using Lemma \ref{lem:many-to-two} and Lemma \ref{lem: Lemma many-to-two product}, we can estimate the terms on the right-hand side above individually. We start with the mixed term
\begin{align*}
    \E[L_nW_n]
    &=\E\Big[\Big(\frac{1}{n\rho^n}\sum_{v \in T_n}|X_v|\Big)\Big(\frac{1}{\rho^n}\sum_{w \in T_n}1\Big)\Big]
    \\
    &=\frac{1}{n\rho^{2n}}\E\Big[\sum_{v,w\in T_n}|X_v|\Big]\\
    &=\frac{1}{n\rho^{2n}}\ExpHat\Big[|\zeta_n^1|\cdot\prod_{v\in\text{skel}(n)}w_{D_{p(v)}}\Big]\\
    &= \Big(\frac{\theta}{\rho}\Big)^2\sum_{k=0}^{n-1}\ProbHat(\tau=k)\Big(\frac{\theta}{\rho^2}\Big)^k\ExpHat\left[\frac{|\zeta_n^1|}{n}\Big| \tau=k\right]+\ProbHat(\tau\geq n)\Big(\frac{\theta}{\rho^2}\Big)^n\ExpHat\left[\frac{|\zeta_n^1|}{n}\Big|\tau\geq n\right]\\
    &=\E\left[\frac{|Y_n|}{n}\right]\left(\Big(\frac{\theta}{\rho}\Big)^2 \sum_{k=0}^{n-1}\ProbHat(\tau=k)\Big(\frac{\theta}{\rho^2}\Big)^k+\ProbHat(\tau\geq n)\Big(\frac{\theta}{\rho^2}\Big)^n\right)\\
    &\xrightarrow[]{n\rightarrow\infty} \ell\cdot\E[W^2].
\end{align*}
Next, we consider the second moment of $L_n$ using once again the many-to-two formula together with Lemma \ref{lem: Lemma many-to-two product}:
\begin{equation}
\begin{aligned}\label{sum:1}
    \E[L_n^2]&=\E\Big[\Big(\frac{1}{n\rho^n}\sum_{v \in T_n}|X_v|\Big)\Big(\frac{1}{\rho^n}\sum_{w \in T_n}|X_w|\Big)\Big]\\
   &=\frac{1}{n^2\rho^{2n}}\E\Big[\sum_{v,w\in T_n}|X_v|\cdot|X_w|\Big]\\
    &=\frac{1}{n^2\rho^{2n}}\ExpHat\Big[|\zeta_n^1|\cdot|\zeta_n^2|\cdot\prod_{v\in\text{skel}(n)}w_{D_{p(v)}}\Big]\\
    &=\Big(\frac{\theta}{\rho}\Big)^2 \sum_{k=0}^{n-1}\ProbHat(\tau=k)\Big(\frac{\theta}{\rho^2}\Big)^k\ExpHat\left[\frac{|\zeta_n^1|}{n}\cdot\frac{|\zeta_n^2|}{n}\Big| \tau=k\right]\\&+\ProbHat(\tau\geq n)\Big(\frac{\theta}{\rho^2}\Big)^n\ExpHat\left[\frac{|\zeta_n^1|}{n}\cdot\frac{|\zeta_n^2|}{n}\Big|\tau\geq n\right].
\end{aligned}
\end{equation}
We will now consider the two terms above separately. Note that if the two spinal particles split after time $n$, then $\zeta_n^1=\zeta_n^2$, from which it follows that they are the same until time $n$. Therefore we have
\begin{align*}
    \lim_{n\rightarrow\infty}\ProbHat(\tau\geq n)\Big(\frac{\theta}{\rho^2}\Big)^n  \ExpHat\left[\frac{|\zeta_n^1|}{n}\cdot\frac{|\zeta_n^2|}{n}\Big|\tau\geq n\right]
    =\lim_{n\rightarrow\infty}\ProbHat(\tau\geq n)\Big(\frac{\theta}{\rho^2}\Big)^n\E\Big[\Big(\frac{|Y_n|}{n}\Big)^2\Big]=0.
\end{align*}
Note that the sequence of expectations above is bounded in view of the exponential moment condition \ref{A3}. Next we split the first sum on the right-hand side of Equation \eqref{sum:1} at index $\sqrt{n}$ and use the fact that if the two spine particles split at time $k<n$, then the random variables $d(\zeta_k^1,\zeta_n^1)$ and $d(\zeta_k^2,\zeta_n^2)$ are independent and distributed according to $|Y_{n-k}|$. Due to Equation \eqref{bound the RW}, we can bound $|\zeta_n^1|$ and $|\zeta_n^2|$ by random walks $\eta_0^{(1)}+\eta_1^{(1)}+\cdots+\eta_{n-1}^{(1)}$ and $\eta_0^{(2)}+\eta_1^{(2)}+\cdots+\eta_{n-1}^{(2)}$ respectively. The increments are i.i.d.~and distributed according to $|Y_1|$. For $k\leq\sqrt{n}$, we have
\begin{align*}
\ExpHat\left[\frac{|\zeta_n^1|}{n}\cdot\frac{|\zeta_n^2|}{n}\Big| \tau=k\right] &\leq \ExpHat\Big[\frac{|\zeta_k^1|+d(\zeta_k^1,\zeta_n^1)}{n}\cdot \frac{|\zeta_k^2|+d(\zeta_k^2,\zeta_n^2)}{n}\Big| \tau=k\Big]\\
&\leq \ExpHat\Big[\frac{|\zeta_k^1|}{n}\cdot \frac{|\zeta_k^2|}{n}\Big| \tau=k\Big] + \ExpHat\Big[\frac{|\zeta_k^1|}{n}\cdot \frac{d(\zeta_k^2,\zeta_n^2)}{n}\Big| \tau=k\Big] \\
&+\ExpHat\Big[\frac{d(\zeta_k^1,\zeta_n^1)}{n}\cdot \frac{|\zeta_k^2|}{n}\Big| \tau=k\Big] + \ExpHat\Big[\frac{d(\zeta_k^{1},\zeta_n^1)}{n}\cdot \frac{d(\zeta_k^2,\zeta_n^2)}{n}\Big| \tau=k\Big].
\end{align*}
Since on the event $\{\tau = k\}$ the spine particles are the same, using Jensen's inequality we obtain
\begin{align*}
\ExpHat\Big[\frac{|\zeta_k^1|}{n}\cdot \frac{|\zeta_k^2|}{n}\Big| \tau=k\Big] \leq \ExpHat\Big[\Big(\frac{\eta_0^{(1)}+\eta_1^{(1)}+\cdots+\eta_{k-1}^{(1)}}{n}\Big)^2 \Big| \tau=k\Big]\leq \frac{k^2}{n^2}\E\big[|Y_1|^2\big].
\end{align*}
The second and third summands can be treated the same way. Note that conditioned on the event $\{\tau = k\}$ the random variables $\zeta_k^1$ and $d(\zeta_k^2,\zeta_n^2)$ are independent and thus we get
\begin{align*}
\ExpHat\Big[\frac{|\zeta_k^1|}{n}\cdot \frac{d(\zeta_k^2,\zeta_n^2)}{n}\Big| \tau=k\Big]&= \ExpHat\left[\frac{|\zeta_k^1|}{n}\Big| \tau=k\right]\ExpHat\left[\frac{d(\zeta_k^2,\zeta_n^2)}{n}\Big| \tau=k\right] \\
&\leq \ExpHat\left[\frac{\sum_{i=0}^{k-1}\eta_i^{(1)}}{n} \Big| \tau=k\right]\ExpHat\left[\frac{\sum_{i=k}^{n-1}\eta_i^{(2)}}{n} \Big| \tau=k\right]\\
&\leq \frac{k(n-k)}{n^2}\E[|Y_1|]^2.
\end{align*}
For the last summand we have  $\ExpHat\Big[\frac{d(\zeta_k^{1},\zeta_n^1)}{n}\cdot \frac{d(\zeta_k^2,\zeta_n^2)}{n}\Big| \tau=k\Big]=\frac{\E[|Y_{n-k}|]^2}{n^2}$. 
Together with Lemma \ref{lem: WLLN helping lemma 1}, we obtain the following upper bound for the  first sum (up to index $\sqrt{n}$) on the right-hand side of Equation \eqref{sum:1}
\begin{align*}
    &\sum_{k=0}^{\sqrt{n}}\ProbHat(\tau=k)\Big(\frac{\theta}{\rho^2}\Big)^k\ExpHat\left[\frac{|\zeta_n^1|}{n}\cdot\frac{|\zeta_n^2|}{n}\Big| \tau=k\right]\\
\leq& \bigg(\frac{n\E\big[|Y_1|^2\big]}{n^2} + \frac{2n^{\frac{3}{2}}\E\big[|Y_1|]^2}{n^2} 
 + \max_{k\leq\sqrt{n}}\mathbb{E}\left[\frac{|Y_{n-k}|}{n}\right]^2\bigg) \sum_{k=0}^{\sqrt{n}}\ProbHat(\tau=k)\Big(\frac{\theta}{\rho^2}\Big)^k\\ &\xrightarrow[]{n\rightarrow\infty} \ell^2 \sum_{k=0}^{\infty}\ProbHat(\tau=k)\Big(\frac{\theta}{\rho^2}\Big)^k.
\end{align*}
Moreover due to Lemma \ref{lemma: decay of splitting time} the remaining sum clearly converges to $0$ as $n\to\infty$:
\begin{align*}
    \sum_{k=\sqrt{n}+1}^{n}\ProbHat(\tau=k)\Big(\frac{\theta}{\rho^2}\Big)^k\ExpHat\left[\frac{|\zeta_n^1|}{n}\cdot\frac{|\zeta_n^2|}{n}\Big| \tau=k\right]\leq \E\big[|Y_1|^2\big]\sum_{k=\sqrt{n}+1}^{\infty}\ProbHat(\tau=k)\Big(\frac{\theta}{\rho^2}\Big)^k\xrightarrow[]{n\rightarrow\infty} 0.
\end{align*}
This implies that the limit of the second moment of $L_n^2$ can be bounded from above by
\begin{align*}
    \limsup_{n\to\infty}\E\big[L_n^2\big] \leq \ell^2\Big(\frac{\theta}{\rho}\Big)^2 \sum_{k=0}^{\infty}\ProbHat(\tau=k)\Big(\frac{\theta}{\rho^2}\Big)^k=\ell^2\E[W^2],
\end{align*}
where in the  second equality we have used Lemma \ref{lemma: decay of splitting time}. Putting the estimations above together, we thus obtain
$$0\leq\limsup_{n\rightarrow\infty}\E[(L_n-\ell W_n)^2]\leq \ell^2\E[W^2]-2\ell^2\E[W^2]+\ell^2\E[W^2]=0$$
which completes the proof.
\end{proof}

\subsection{Maximal displacement}\label{subsec: Maximal Displacement}
In order to extend the convergence of the mean displacement of the empirical distribution in $L^2$ from Proposition \ref{prop:weak-lln} to almost sure convergence, we first estimate the maximal displacement among all branching particles at time $n$. We define
$$\maxdis_n=\max_{v\in T_n}|X_v|.$$
While on the additive group of integers and of reals, convergence results for the maximal displacement of particles in a branching process are well understood, see \cite{FirstBirthProbKing,FirstBirthProbBigg,FirstBirthProbHamm}), for our purposes it suffices to show that the maximal displacement can grow at most linearly in $n$. Let us define
\begin{align*}
 \Theta_n(t)=\E\Big[\sum_{v\in T_n}\e^{t|X_v|}\Big]\quad \text{and}\quad \vartheta_n(t)=\E\big[\e^{t|Y_n|}\big].
\end{align*}
From Lemma \ref{lem: many-to-one} follows  that
$$\Theta_n(t)=\rho^n\vartheta_n(t)\leq(\rho\vartheta_1(t))^n.$$
Recall the assumption \ref{A3} on the transition kernel that there exists a real number $t>0$ such that $\vartheta_1(t)<\infty$.
\begin{prop}\label{prop:maxdis}
Under assumptions \ref{A1}-\ref{A5}, there exists $a>0$ such that
$$\limsup_{n\rightarrow\infty}\frac{\maxdis_n}{n}\leq a,\quad \text{almost surely}.$$
\end{prop}
\begin{proof}
Let us choose some $a>0$, whose value will be specified later and $t>0$ such that $\vartheta_1(t)<\infty$. Then we have
\begin{align*}
\Theta_n(t)=\E\Big[\sum_{v\in T_n}\e^{t|X_v|}\Big]
\geq \E\big[\e^{t\maxdis_n}\big]
\geq \E\big[\e^{t\maxdis_n}\mathbbm{1}_{\{\maxdis_n\geq na\}}\big]
\geq \e^{tna}\Prob(\maxdis_n\geq na),
\end{align*}
which implies
$$\Prob(\maxdis_n\geq na)\leq \Big(\frac{\rho\vartheta_1(t)}{\e^{ta}}\Big)^n.$$
If we now choose $a>0$ such that
$$\Big(\frac{\rho\vartheta_1(t)}{\e^{ta}}\Big)<1,$$
then the claim follows from an application of the Borel-Cantelli lemma.
\end{proof}

\subsection{Strong law of large numbers}

We are now finally in position to prove Theorem \ref{thm:first_main}. The proof relies on the following lemma, which allows us to decompose $L_n$ into a sum of i.i.d. random variables, up to a small error term, that we can control. Recall the definition of the mean displacement
$$L_n = \frac{1}{n\rho^n}\sum_{v\in T_n} |X_v|.$$
For $k<n$ and $w\in T_k$, we write $T_n^{(w)}$ for all vertices at level $n$ which have $w$ as ancestor:
$$T_n^{(w)} = \lbrace v\in T_n\,|\, p^{n-k}(v) = w\rbrace.$$

\begin{lemma}\label{lem:decomp}
Under assumptions \ref{A1}-\ref{A5}, for every  $n\in\N$ and $k\in\{1,...,n-1\}$ there exist i.i.d. random variables $\big(L_{n,k}^{(w)}\big)_{w\in T_k}$ with the same distribution as $L_{n-k}$, i.e.
    $$L_{n,k}^{(w)}\sim L_{n-k},$$
    for all $w\in T_k$, such that
    $$\Big|L_n-\frac{n-k}{n\rho^k}\sum_{w\in T_k}L_{n,k}^{(w)}\Big|\leq\frac{1}{n\rho^n}\sum_{w\in T_k}|X_w|\cdot|T_n^{(w)}|.$$
\end{lemma}
\begin{proof}
   Fix some $k<n$. We rewrite $L_n$ as follows:
    $$L_n=\frac{1}{n\rho^n}\sum_{v\in T_n}|X_v|=\frac{1}{n\rho^n}\sum_{w\in T_k}\sum_{v\in T_n^{(w)}}|X_v|,$$
    which together with the triangle inequality implies that for all $w\in T_k$ and $v\in T_n^{(w)}$ 
$$|X_v|\leq |X_w|+d(X_w,X_v),$$
and this yields
$$L_n\leq \frac{1}{n\rho^n}\sum_{w\in T_k}\sum_{v\in T_n^{(w)}}\Big(|X_w|+d(X_w,X_v)\Big).$$
The Markov property of the branching process together with assumption \ref{A4} imply that the random variables 
    $$L_{n,k}^{(w)}:=\frac{1}{(n-k)\rho^{n-k}}\sum_{v\in T_n^{(w)}}d(X_w,X_v)$$
 are i.i.d.~and have the same distribution as $L_{n-k}.$ Thus
    $$L_n\leq \frac{n-k}{n\rho^k}\sum_{w\in T_k}L_{n,k}^{(w)} +\frac{1}{n\rho^n}\sum_{w\in T_k}|X_w|\cdot|T_n^{(w)}|.$$
    By the reverse triangle inequality we also obtain
    $$L_n\geq \frac{1}{n\rho^n}\sum_{w\in T_k}\sum_{v\in T_n^{(w)}}\Big(d(X_w,X_v)-|X_w|\Big),$$
    from which it follows by the same reasoning as above
    $$L_n\geq \frac{n-k}{n\rho^k}\sum_{w\in T_k}L_{n,k}^{(w)} -\frac{1}{n\rho^n}\sum_{w\in T_k}|X_w|\cdot|T_n^{(w)}|,$$
    completing the proof.
\end{proof}

\subsubsection{Proof of Theorem \ref{thm:first_main}}
\begin{proof}
Applying Lemma \ref{lem:decomp} for $k=\sqrt{n}$ we obtain
    $$\bigg|L_n-\frac{n-\sqrt{n}}{n\rho^{\sqrt{n}}}\sum_{w\in T_{\sqrt{n}}}L_{n,\sqrt{n}}^{(w)}\bigg|\leq\frac{1}{n\rho^n}\sum_{w\in T_{\sqrt{n}}}|X_w|\cdot|T_n^{(w)}|,$$
    where the random variables $(L_{n,\sqrt{n}}^{(w)})_{w\in T_{\sqrt{n}}}$ are i.i.d.~with $L_{n,\sqrt{n}}^{(w)}\overset{\mathcal{D}}{=}L_{n-\sqrt{n}}.$ We show that the upper bound in the above inequality converges almost surely to $0$. Since it holds
$|X_w|\leq \maxdis_{\sqrt{n}}$ for all $w\in T_{\sqrt{n}}$, we obtain
$$\frac{1}{n\rho^n}\sum_{w\in T_{\sqrt{n}}}|X_w|\cdot|T_n^{(w)}|\leq\frac{\maxdis_{\sqrt{n}}}{n\rho^n}\sum_{w\in T_{\sqrt{n}}}|T_n^{(w)}|=\frac{\maxdis_{\sqrt{n}}\cdot W_n}{n}.$$
By Proposition \ref{prop:maxdis} we have
$$\lim_{n\to\infty}\frac{\maxdis_{\sqrt{n}}\cdot W_n}{n} = \lim_{n\to\infty} \frac{\maxdis_{\sqrt{n}}}{\sqrt{n}}\cdot W_n\cdot\frac{1}{\sqrt{n}}\leq \lim_{n\to\infty} a\cdot W_n\cdot\frac{1}{\sqrt{n}} = 0,$$
almost surely, thus 
\begin{align}\label{eq:comp-to-sum}
\lim_{n\to\infty}  \bigg( L_n-\frac{n-\sqrt{n}}{n\rho^{\sqrt{n}}}\sum_{w\in T_{\sqrt{n}}}L_{n,\sqrt{n}}^{(w)}\bigg)=0, \quad \text{almost surely}.
\end{align}
So it remains to show that the sum $\frac{n-\sqrt{n}}{n\rho^{\sqrt{n}}}\sum_{w\in T_{\sqrt{n}}}L_{n,\sqrt{n}}^{(w)}$ of i.i.d.~random variables converges almost surely to the desired limit $\ell\cdot W$. For $w\in T_{\sqrt{n}}$, we write 
\begin{align*}  
\ell_n:=\E[L_n]\quad \text{and}\quad W_n^{(w)}:=\frac{|T_n^{(w)}|}{\rho^{n-\sqrt{n}}}.
\end{align*}
Lemma \ref{lem: many-to-one} together with Kingman's subadditive ergodic theorem implies that the expectation $\ell_n$ of $L_n$ converges to $\ell$ as $n\to\infty$: 
$$\ell_n = \E[L_n] = \frac{1}{n\rho^n}\E\Big[\sum_{v\in T_n}|X_v|\Big] = \frac{\rho^n}{n\rho^n}\E\big[|Y_n|\big] = \E\left[\frac{|Y_n|}{n}\right] \xrightarrow[]{n\rightarrow\infty} \ell.$$   
We then have
\begin{align*}
 \E\Big[\Big(\frac{1}{\rho^{\sqrt{n}}}\sum_{w\in T_{\sqrt{n}}}L_{n,\sqrt{n}}^{(w)}&-\ell_{n-\sqrt{n}} W_n^{(w)}\Big)^2 \Big]
        =\frac{1}{\rho^{2\sqrt{n}}}\E\Big[\sum_{w\in T_{\sqrt{n}}}\big(L_{n,\sqrt{n}}^{(w)}-\ell_{n-\sqrt{n}}W_n^{(w)}\big)^2 \Big]\\
        &+\frac{1}{\rho^{2\sqrt{n}}}\E\Big[\sum_{\substack{v,w\in T_{\sqrt{n}}\\v\neq w}}\big(L_{n,\sqrt{n}}^{(w)}-\ell_{n-\sqrt{n}}W_n^{(w)}\big)\big(L_{n,\sqrt{n}}^{(v)}-\ell_{n-\sqrt{n}}W_n^{(v)}\big)\Big].
    \end{align*} 
Proposition \ref{prop:weak-lln} together with Lemma \ref{lem: many-to-one} implies the existence of a constant $C>0$ such that
    $$\frac{1}{\rho^{2\sqrt{n}}}\E\Big[\sum_{w\in T_{\sqrt{n}}}\big(L_{n,\sqrt{n}}^{(w)}-\ell_{n-\sqrt{n}}W_n^{(w)}\big)^2\Big]=\frac{1}{\rho^{\sqrt{n}}}\E\Big[(L_{n-\sqrt{n}}-\ell_{n-\sqrt{n}}W_{n-\sqrt{n}}\big)^2\Big]\leq\frac{C}{\rho^{\sqrt{n}}} $$
    and since the random variables $L_{n,\sqrt{n}}^{(w)}$ and $L_{n,\sqrt{n}}^{(v)}$  as well as $W_n^{(w)}$ and $W_n^{(v)}$are independent for $v\neq w$, we get
\begin{align*}
\frac{1}{\rho^{2\sqrt{n}}}&\E\Big[\sum_{\substack{v,w\in T_{\sqrt{n}}\\v\neq w}}\big(L_{n,\sqrt{n}}^{(w)}-\ell_{n-\sqrt{n}}W_n^{(w)}\big)\big(L_{n,\sqrt{n}}^{(v)}-\ell_{n-\sqrt{n}}W_n^{(v)}\big)\Big|T_{\sqrt{n}}\Big]= \\
    &=\frac{1}{\rho^{2\sqrt{n}}}\sum_{\substack{v,w\in T_{\sqrt{n}}\\v\neq w}}\E\Big[ \big(L_{n,\sqrt{n}}^{(w)}-\ell_{n-\sqrt{n}}W_n^{(w)}\big)\big(L_{n,\sqrt{n}}^{(v)}-\ell_{n-\sqrt{n}}W_n^{(v)}\big)\big| T_{\sqrt{n}}\Big] \\
    &=\frac{1}{\rho^{2\sqrt{n}}}\sum_{\substack{v,w\in T_{\sqrt{n}}\\v\neq w}}\E\Big[ \big(L_{n,\sqrt{n}}^{(w)}-\ell_{n-\sqrt{n}} W_n^{(w)}\big)\big|T_{\sqrt{n}}\big]\E\big[\big(L_{n,\sqrt{n}}^{(v)}-\ell_{n-\sqrt{n}}W_n^{(v)}\big)\big|T_{\sqrt{n}}\Big]=0.
\end{align*}
By taking the expectation over the conditional expectation above, we obtain 
$$\frac{1}{\rho^{2\sqrt{n}}}\E\Big[\sum_{\substack{v,w\in T_{\sqrt{n}}\\v\neq w}}\big(L_{n,\sqrt{n}}^{(w)}-\ell_{n-\sqrt{n}}W_n^{(w)}\big)\big(L_{n,\sqrt{n}}^{(v)}-\ell_{n-\sqrt{n}}W_n^{(v)}\big)\Big]=0,$$
and therefore
    $$\E\Big[\Big(\frac{1}{\rho^{\sqrt{n}}}\sum_{w\in T_{\sqrt{n}}}L_{n,\sqrt{n}}^{(w)}-\ell_{n-\sqrt{n}}W_n^{(w)}\Big)^2\Big]\leq\frac{C}{\rho^{\sqrt{n}}}.$$
Since $\sum_{n=1}^{\infty}\frac{C}{\rho^{\sqrt{n}}}<\infty$,
in view of Markov's inequality together with Borell-Cantelli lemma we get
    $$\lim_{n\rightarrow\infty}\frac{1}{\rho^{\sqrt{n}}}\sum_{w\in T_{\sqrt{n}}}\Big(L_{n,\sqrt{n}}^{(w)}-\ell_{n-\sqrt{n}}W_n^{(w)}\Big) = 0, \quad \text{almost surely.}$$
Furthermore
$$\frac{1}{\rho^{\sqrt{n}}}\sum_{w\in T_{\sqrt{n}}}l_{n-\sqrt{n}}W_n^{(w)}=\ell_{n-\sqrt{n}}W_n\converges \ell\cdot W, \quad \text{almost surely},$$
and thus
$$\lim_{n\rightarrow\infty}\frac{n-\sqrt{n}}{n\rho^{\sqrt{n}}}\sum_{w\in T_{\sqrt{n}}}L_{n,\sqrt{n}}^{(w)}=\ell\cdot W,$$
which together with Equation \eqref{eq:comp-to-sum} implies
$$L_n=\frac{1}{n}\int |x|\,dM_n(x)\converges \ell\cdot W, \quad \text{almost surely},$$
and this proves the claim in Theorem \ref{thm:first_main}.
\end{proof}

\section{Central limit theorem}\label{sec:clt}

This section is dedicated to the proof of Theorem \ref{thm:second_main}, which relies   on Stam \cite{ConjHarris}, where a central limit theorem for branching random walks on $\Z$ is proven. Whereas on $\Z$ many computations can be done explicitly due to the additive nature of the group of integers, this is not the case anymore on general transitive graphs and more abstract tools are needed. Our proof is based on characteristic functions and on the Levy's continuity theorem. The proof of Theorem \ref{thm:second_main} is done in the following three main steps:
\begin{itemize}
\setlength\itemsep{0em}
\item We first prove in Theorem \ref{thm: CLT mean square} that the characteristic function of the rescaled empirical distribution converges in mean square distance to $W$ times the characteristic function of the standard normal distribution.
\item We then prove in Section \ref{subsec: Convergence in probability} that the rescaled empirical distributions in fact converge in probability to $W$ times a standard normal distribution independent of $W$. We make once again use of the many-to-two formula stated in Lemma \ref{lem:many-to-two}.
\item Finally, the proof of  Theorem \ref{thm:second_main} is then based on the Levy's continuity theorem.
\end{itemize}
For the rest of this section, we always suppose that the assumptions \ref{A1}-\ref{A5} hold, and in particular the underlying random walk $(Y_n)_{n\in \N}$ obeys a central limit theorem, i.e.~it exists $\sigma > 0$ such that 
$$\frac{|Y_n| - n\ell}{\sigma\sqrt{n}} \CID \mathcal{N}(0,1), \quad \text{as }n\to\infty.$$ This is in general not the case, and we conclude with an example where such a central limit theorem holds.
 
We define the characteristic function of the empirical distribution $(M_n)_{n\in \mathbb{N}}$ and of the random walk $(Y_n)_{n\in \N}$ respectively as: for $t\in \R$
\begin{align*}
\Psi_n(t) &= \int \e^{it|x|}\, dM_n(x) = \frac{1}{\rho^n}\sum_{v\in T_n} \e^{i t|X_v|}  \\
\varphi_n(t) &= \mathbb{E}\Big[\e^{it|Y_n|}\Big].
\end{align*}
We write $\varphi$ for the characteristic function of a standard Gaussian random variable, i.e. for $t\in \R$
\begin{align*}
\varphi(t) = \e^{-\frac{t^2}{2}}.
\end{align*}
The assumption that the underlying Markov chain obeys a central limit theorem implies 
$$\lim_{n\to\infty}\E\Big[\e^{it\frac{|Y_n|-n\ell}{\sigma\sqrt{n}}}\Big] = \varphi(t),
\quad \text{for } t\in \R.$$
We also have
$$\e^{-it\frac{\sqrt{n}\ell}{\sigma}}\varphi_n\Big(\frac{t}{\sigma\sqrt{n}}\Big) = \e^{-it\frac{\sqrt{n}\ell}{\sigma}}\E\Big[\e^{i\frac{t}{\sigma\sqrt{n}}|Y_n|}\Big]=\E\Big[\e^{it\frac{|Y_n|-n\ell}{\sigma\sqrt{n}}}\Big],
$$
which in turn implies
$$\lim_{n\to\infty}\e^{-it\frac{\sqrt{n}\ell}{\sigma}}\varphi_n\Big(\frac{t}{\sigma\sqrt{n}}\Big) = \varphi(t), \quad t\in \R.$$

\subsection{Convergence in mean}\label{subsec: Convergence in mean}
In this section we prove that the random function $\Psi_n(t/\sigma\sqrt{n})$ converges in mean square distance to $W\cdot\varphi(t)$.
\begin{theorem}\label{thm: CLT mean square}
Under assumptions \ref{A1}-\ref{A5}, for any $t\in\R$ it holds
\begin{align*}
\lim_{n\to\infty} \mathbb{E}\left[\Big|\Psi_n\Big(\frac{t}{\sigma\sqrt{n}}\Big) - W_n\varphi_n\Big(\frac{t}{\sigma \sqrt{n}}\Big)\Big|^2\right]  = 0.
\end{align*}
\end{theorem}
\begin{proof}
The key idea of the proof is to apply the many-to-two formula from Lemma \ref{lem:many-to-two}. For the rest of proof we set $t_n := t/(\sigma \sqrt{n})$ for simplicity. By rewriting the square under the expectation, we obtain:
\begin{align*}
 |\Psi_n(t_n) - W_n\varphi_n(t_n)|^2 &= \Big{|}\Psi_n(t_n) - \frac{|T_n|}{\rho^n}\varphi_n(t_n)\Big{|} ^2\\
&=\Big{|}\frac{1}{\rho^n}\sum_{v\in T_n} \Big{(}\e^{it_n|X_v|} - \varphi_n(t_n)\Big{)}\Big{|}^2 \\
& =\frac{1}{\rho^{2n}}\sum_{v,w\in T_n} \Big{(}\e^{it_n|X_v|} - \varphi_n(t_n)\Big{)}\overline{\Big{(}\e^{it_n|X_w|} - \varphi_n(t_n)\Big{)}}.
    \end{align*}
Since the right-hand side in the previous equation is a sum over all pairs $v,w\in T_n$ of a function depending on the positions $X_v$ and $X_w$, we can apply once again Lemma \ref{lem:many-to-two} which together with Lemma \ref{lem: Lemma many-to-two product} gives
\begin{equation}
\begin{aligned}\label{eq: expectation_1_CLT}
        &\mathbb{E}\big[|\Psi_n(t_n) - W_n\varphi_n(t_n)|^2\big] \\
       & =\frac{1}{\rho^{2n}}\ExpHat\Big[ \Big(\e^{it_n|\zeta_n^1|} - \varphi_n(t_n)\Big)\overline{\Big(\e^{it_n|\zeta_n^2|} - \varphi_n(t_n)\Big)}  \prod_{v\in\textit{skel}(n)} w_{D_{p(v)}} \Big]\\
            &=\ProbHat(\tau\geq n)\Big(\frac{\theta}{\rho^2}\Big)^n\ExpHat\Big[\Big{(}\e^{it_n|\zeta_n^1|} - \varphi_n(t_n)\Big{)}\overline{\Big{(}\e^{it_n|\zeta_n^2|} - \varphi_n(t_n)\Big{)}}\Big|\tau\geq n\Big]
       \\  &+\Big(\frac{\theta}{\rho}\Big)^2\sum_{k=0}^{n-1}\ProbHat(\tau=k)\Big(\frac{\theta}{\rho^2}\Big)^k\ExpHat\Big[\Big{(}\e^{it_n|\zeta_n^1|} - \varphi_n(t_n)\Big{)}\overline{\Big{(}\e^{it_n|\zeta_n^2|} - \varphi_n(t_n)\Big{)}}\Big| \tau=k\Big].
\end{aligned}
\end{equation}
We now estimate the two terms on the right-hand side of Equation \eqref{eq: expectation_1_CLT}. Since the characteristic functions are all bounded by $1$ we get for the first term  
\begin{align*}
\Big|\ProbHat(\tau\geq n)&\Big(\frac{\theta}{\rho^2}\Big)^n\ExpHat\Big[\Big{(}\e^{it_n|\zeta_n^1|} - \varphi_n(t_n)\Big{)}\overline{\Big{(}\e^{it_n|\zeta_n^2|} - \varphi_n(t_n)\Big{)}}\Big|\tau\geq n\Big]\Big| \\
&\leq 4\ProbHat(\tau\geq n)\Big(\frac{\theta}{\rho^2}\Big)^n\converges 0.
\end{align*}
The limit holds due to Lemma \ref{lemma: decay of splitting time}. We have the bound  
$$\left|\ExpHat\Big[\Big{(}\e^{it_n|\zeta_n^1|} - \varphi_n(t_n)\Big{)}\overline{\Big{(}\e^{it_n|\zeta_n^2|} - \varphi_n(t_n)\Big{)}}\Big| \tau=k\Big]\right|\leq 4,$$ 
which is uniform in $k,n\in\N$ and $t\in\R$. Moreover, it holds
\begin{align*}
\sum_{k=0}^{n-1}\ProbHat(\tau=k)&\Big(\frac{\theta}{\rho^2}\Big)^k\ExpHat\Big[\Big(\e^{it_n|\zeta_n^1|} - \varphi_n(t_n)\Big)\overline{\Big(\e^{it_n|\zeta_n^2|} - \varphi_n(t_n)\Big)}\Big| \tau=k\Big]\\
& \sum_{k=0}^{\infty}\ProbHat(\tau=k)\Big(\frac{\theta}{\rho^2}\Big)^k\ExpHat\Big[\Big(\e^{it_n|\zeta_n^1|} - \varphi_n(t_n)\Big)\overline{\Big(\e^{it_n|\zeta_n^2|} - \varphi_n(t_n)\Big)}\Big| \tau=k\Big]\mathbf{1}\{k<n\}.
\end{align*}
Thus, in order to bound the second term on the right-hand side of Equation \eqref{eq: expectation_1_CLT}, by the dominated convergence theorem, 
it suffices to show that for every $t\in\R$ and $k\in\N$ it holds that
\begin{equation}
\begin{aligned}\label{eq: limit_expectation_1_CLT}
\lim_{n\to\infty}\ExpHat\Big[\Big{(}\e^{it_n|\zeta_n^1|} - \varphi_n(t_n)\Big{)}\overline{\Big{(}\e^{it_n|\zeta_n^2|} - \varphi_n(t_n)\Big{)}}\Big| \tau=k\Big] = 0.
\end{aligned}
\end{equation}
Note that on the event $\{\tau = k\}$, $|\zeta_n^1|$ and $
|\zeta_n^2|$ are distributed as $|Y_n|$, and so a simple computation yields 
\begin{align*}
    &\ExpHat\Big[\Big{(}\e^{it_n|\zeta_n^1|} - \varphi_n(t_n)\Big{)}\overline{\Big{(}\e^{it_n|\zeta_n^2|} - \varphi_n(t_n)\Big{)}}\Big| \tau=k\Big] \\
    &= \ExpHat\Big[\e^{it_n(|\zeta_n^1|-|\zeta_n^2|)}\Big|\tau =k\Big]- \ExpHat\Big[\e^{it_n|\zeta_n^1|}\Big|\tau=k\Big]\overline{\varphi_n(t_n)} - \varphi_n(t_n)\ExpHat\Big[\e^{-it_n|\zeta_n^2|}\Big|\tau = k \Big] +|\varphi_n(t_n)|^2\\
    &= \ExpHat\Big[\e^{it_n(|\zeta_n^1|-|\zeta_n^2|)}\Big|\tau = k \big] -2\varphi_n(t_n)\overline{\varphi_n(t_n)}+|\varphi_n(t_n)|^2
    \\
    &=\ExpHat\Big[\e^{it_n(|\zeta_n^1|-|\zeta_n^2|)}\Big|\tau = k \Big]  - |\varphi_n(t_n)|^2.
\end{align*} 
The key observation needed in order to estimate this term is that after the two spine particles have split, they move independent of each other, and therefore the term 
$\ExpHat\big[\e^{it_n|\zeta_n^1|}\e^{-it_n|\zeta_n^2|} \big| \tau=k\big]$
stays close to $|\varphi_n(t_n)|^2$ as long as $k$ is much smaller then $n$. Moreover, the probability that spine particles split at later times decays exponentially in view of Lemma \ref{lemma: decay of splitting time}.
We first estimate the exponent of the exponential inside the above expectation by using the triangle and reverse triangle inequality. For $k\leq n$ it holds 
\begin{align*}
|\zeta_n^1|-|\zeta_n^2|\leq d(\zeta_k^1,\zeta_n^1) + |\zeta_k^1| - d(\zeta_k^2,\zeta_n^2) + |\zeta_k^2| 
 = d(\zeta_k^1,\zeta_n^1) - d(\zeta_k^2,\zeta_n^2) + (|\zeta_k^1|+|\zeta_k^2|),
\end{align*}
and similarly
\begin{align*}
       |\zeta_n^1|-|\zeta_n^2|\geq d(\zeta_k^1,\zeta_n^1) - |\zeta_k^1| - d(\zeta_k^2,\zeta_n^2) - |\zeta_k^2| 
        = d(\zeta_k^1,\zeta_n^1) - d(\zeta_k^2,\zeta_n^2) - (|\zeta_k^1|+|\zeta_k^2|).
\end{align*}
On the event $\lbrace \tau = k \rbrace$, for some $k\leq n$ we set $\zeta_k =  \zeta_k^1 = \zeta_k^2$ to be the random position both particles are located at the time $\tau$ of splitting. The above  estimates imply that on the event $\lbrace\tau = k\rbrace$ for $k\leq n$ the remainder $ R_{n,k} :=   |\zeta_n^1|-|\zeta_n^2| -d(\zeta_k^1,\zeta_n^1) + d(\zeta_k^2,\zeta_n^2)$ is bounded by
\begin{align*}
|R_{n,k}|\leq 2|\zeta_k|.
\end{align*}
Since on the event $\lbrace\tau = k\rbrace$ for $k\leq n$ the random variables  $d(\zeta_k^1,\zeta_n^1)$ and $d(\zeta_k^2,\zeta_n^2)$ are independent and distributed as $|Y_{n-k}|$, the following holds:
\begin{align*}
    \ExpHat\Big[\e^{it_n\big(d(\zeta_k^1,\zeta_n^1) - d(\zeta_k^2,\zeta_n^2) \big)}\Big|\tau = k\Big]  = |\varphi_{n-k}(t_n)|^2.
\end{align*}
We therefore obtain
\begin{align*}
      &\ExpHat\Big[ \e^{it_n(|\zeta_n^1|-|\zeta_n^2|)}\Big|\tau = k\Big] =\\
      &=\ExpHat\Big[ e^{it_n(d(\zeta_k,\zeta_n^1)-d(\zeta_k,\zeta_n^2) )}\Big|\tau = k\Big] + \ExpHat\Big[ e^{it_n(d(\zeta_k,\zeta_n^1)-d(\zeta_k,\zeta_n^2))}\big{(} e^{it_nR_{n,k}}-1\big{)}\Big|\tau = k\Big] \\
      &= |\varphi_{n-k}(t_n)|^2 + \ExpHat\Big[ e^{it_n(d(\zeta_k,\zeta_n^1)-d(\zeta_k,\zeta_n^2))}\big{(} e^{it_nR_{n,k}}-1\big{)}\Big|\tau = k\Big] ,
\end{align*} 
which together with the inequality $|\e^{it}-1|\leq |t|$ for $ t\in\R$
yields 
\begin{align*}
    \Big|\ExpHat\Big[ e^{it_n\big(d(\zeta_k,\zeta_n^1)-d(\zeta_k,\zeta_n^2)\big )}\big{(} e^{it_nR_{n,k}}-1\big{)}|\tau = k]\Big|&\leq \ExpHat\Big[\big| e^{it_nR_{n,k}}-1\big|\Big|\tau = k\Big] \\
    &\leq \ExpHat [|t_n|(|\zeta_k^1|+|\zeta_k^2|)]\\
&= \ExpHat \big[ 2|t_n|\cdot|\zeta_k|\big] =\E\big[2|t_n|\cdot|Y_k|\big].
\end{align*}
Thus, the absolute value of the  terms of the sequence in Equation \eqref{eq: limit_expectation_1_CLT} can be bounded by
\begin{align*}
\Big|\ExpHat&\Big[\Big{(} \e^{it_n|\zeta_n^1|} - \varphi_n(t_n)\Big{)}\overline{\Big{(}\e^{it_n|\zeta_n^2|} - \varphi_n(t_n)\Big{)}}\Big| \tau=k\Big]\Big|
\\
&\leq |\varphi_{n-k}(t_n)|^2 + \Big|\ExpHat\Big[ e^{it_n(d(\zeta_k^1,\zeta_n^1)-d(\zeta_k^2,\zeta_n^2))}\Big( e^{it_n R_{n,k}}-1\Big)\Big]\Big| - |\varphi_n(t_n)|^2    
\\
&\leq \big| |\varphi_{n-k}(t_n)|^2  - |\varphi_n(t_n)|^2\big|+2\E\big[|t_n|\cdot|Y_k|\big].
\end{align*}
In view of Lemma \ref{lem: CLT helping lemma 2}, the first term of the right-hand side of the previous equation converges to $0$. The second term converges to $0$ since
$$
\lim_{n\to\infty}|t_n| = \lim_{n\to\infty}\frac{|t|}{\sigma\sqrt{n}} = 0,
$$
and so the proof is finished.
\end{proof}

\subsection{Convergence in probability}\label{subsec: Convergence in probability}

We show next, by using Theorem \ref{thm: CLT mean square}, how to prove the almost sure convergence of the characteristic functions of the rescaled empirical distributions over averages on intervals on subsequences. This result together with the characterization of convergence in probability via subsequences  and Levy's continuity theorem in Theorem \ref{thm: Levy} will be used in order to conclude the proof of Theorem \ref{thm:second_main}.

\begin{lemma}\label{lem: CLT helping lemma 3}
Under assumptions \ref{A1}-\ref{A5}, it holds
\begin{align*}
\lim_{n\to\infty}\E \left[\int_{-\infty}^\infty \e^{-t^2}\Big|\e^{-it\frac{\sqrt{n}\ell}{\sigma}}\Psi_n\Big( \frac{t}{\sigma\sqrt{n}}\Big) - W\varphi(t)\Big|^2\, dt\right] = 0.
\end{align*}
Moreover for every increasing sequence $(a_n)_{n\in \N}$ there exists an increasing subsequence   $(a_{n_k})_{k\in \N}$ such that, for every $y\in\R$
\begin{align*}
 \lim_{n\to\infty}\int_{0}^y \e^{-it\frac{\sqrt{a_{n_k}}\ell}{\sigma}}\Psi_{a_{n_k}}\Big( \frac{t}{\sigma\sqrt{a_{n_k}}}\Big) \,dt =
W\cdot\int_0^y   \varphi(t) \,dt\quad \text{almost surely}.
\end{align*}
\end{lemma}
\begin{proof}
First note that 
$$ f\mapsto \E \left[\int_{-\infty}^\infty \frac{\e^{-t^2}}{\sqrt{2\pi}}\cdot|f(t)|^2\, dt\right]^{\frac{1}{2}}$$
is a norm, and thus fulfills the triangle inequality. Since $|Y_n|$ fulfills a central limit theorem, and since the population martingale $W_n$ converges in $L^2$, it holds 
\begin{align*}
\lim_{n\to\infty}W_n\e^{-it\frac{\sqrt{n}\ell}{\sigma}}\varphi_n\Big( \frac{t}{\sigma\sqrt{n}}\Big) = W\cdot\varphi(t), \quad t\in\R
\end{align*}
in $L^2$. We have
    \begin{align*}
        \E & \left[\int_{-\infty}^\infty \e^{-t^2}\Big|\e^{-it\frac{\sqrt{n}\ell}{\sigma}}\Psi_n \Big( \frac{t}{\sigma\sqrt{n}}\Big) - W\varphi(t)\Big|^2\,  dt \right]^{\frac{1}{2}} \\ & \leq \E \left[\int_{-\infty}^\infty \e^{-t^2}\Big|\Psi_n\Big( \frac{t}{\sigma\sqrt{n}}\Big) - W_n\varphi_n\Big( \frac{t}{\sigma\sqrt{n}}\Big)\Big|^2\, dt \right]^{\frac{1}{2}}    \\
         &\hspace{1cm}+ \E \left[\int_{-\infty}^\infty \e^{-t^2}\Big|W_n\e^{-it\frac{\sqrt{n}\ell}{\sigma}}\varphi_n\Big( \frac{t}{\sigma\sqrt{n}}\Big) -  W\varphi(t)\Big|^2\, dt\right]^{\frac{1}{2}}\\
         & = \left(\int_{-\infty}^\infty \e^{-t^2}\E\left[\Big|\Psi_n\Big( \frac{t}{\sigma\sqrt{n}}\Big) - W_n\varphi_n\Big( \frac{t}{\sigma\sqrt{n}}\Big)\Big|^2\right]\, dt\right)^{\frac{1}{2}}
         \\
         &\hspace{1cm}+\left(\int_{-\infty}^\infty \e^{-t^2}\E\left[\Big|W_n\e^{-it\frac{\sqrt{n}\ell}{\sigma}}\varphi_n\Big( \frac{t}{\sigma\sqrt{n}}\Big) -  W\varphi(t) \Big|^2\right]\, dt\right)^{\frac{1}{2}} \converges 0,
    \end{align*}
where the first equality follows from the triangle inequality, the second equality follows from Fubini's theorem, and the third equality follows from  Theorem \ref{thm: CLT mean square} and Lemma \ref{lem: L2 population martingale}, and so the first claim is proven. 

Due to \cite[Corollary 6.13, Page 151]{Klenke}, convergence in probability implies almost sure convergence along a subsequence. Thus, 
for the second statement, there exists an increasing subsequence  $(a_{n_k})_{k\in \N}$ such that
      \begin{align*}
        \lim_{k\to\infty}\int_{-\infty}^\infty \e^{-t^2}\Big|\e^{-it\frac{\sqrt{ a_{n_k}}\ell}{\sigma}}\Psi_{a_{n_k}}\Big( \frac{t}{\sigma\sqrt{a_{n_k}}}\Big) - W\varphi(t)\Big|^2 \, dt = 0, \quad \text{almost surely}.
    \end{align*}
Due to Jensen's inequality  we have
 \begin{align*}
        &\int_{-\infty}^\infty \e^{-t^2}\Big|\e^{-it\frac{\sqrt{a_{n_k}}\ell}{\sigma}}\Psi_{a_{n_k}}\Big( \frac{t}{\sigma\sqrt{a_{n_k}}}\Big) - W\varphi(t)\Big|^2\, dt  \\&\geq \frac{1}{\sqrt{\pi}}\Big(  \int_{-\infty}^\infty \e^{-t^2}\Big|\e^{-it\frac{\sqrt{a_{n_k}}\ell}{\sigma}}\Psi_{a_{n_k}}\Big( \frac{t}{\sigma\sqrt{a_{n_k}}}\Big) - W\varphi(t)\Big|\, dt\Big)^2,
    \end{align*}
and since the left-hand side converges to $0$ almost surely we obtain 
\begin{align*}
\lim_{k\to\infty}\int_0^y \Big|\e^{-it\frac{\sqrt{a_{n_k}}\ell}{\sigma}}\Psi_{a_{n_k}}\Big( \frac{t}{\sigma\sqrt{a_{n_k}}}\Big) - W\varphi(t)\Big|\, dt = 0 \quad\text{almost surely,}
\end{align*}
for any $y\in \R.$
This further implies that
    \begin{align*}
            \lim_{k\to\infty}\int_{0}^y \e^{-it\frac{\sqrt{a_{n_k}}\ell}{\sigma}}\Psi_{a_{n_k}}\Big( \frac{t}{\sigma\sqrt{a_{n_k}}}\Big) \,dt
        = W \cdot\int_0^y  \varphi(t) \,dt\quad \text{almost surely,}
    \end{align*}
for any $y\in\R$, and this completes the proof of the second claim.
\end{proof}

\subsection{Proof of Theorem \ref{thm:second_main}}
\begin{proof}
Recall the definition of $M_n^*$ as defined in (\ref{eq:mstar}), which we will rewrite as
$$M_n^* := \frac{1}{\rho^n}\sum_{v\in T_n}\delta_{\frac{|X_v|-n\ell}{\sigma\sqrt{n}}}.$$ 
Note that this is just the empirical distribution $M_n$ as defined in Equation \eqref{eq:emp-distr} transformed by the map 
$$G\to \R:x\mapsto \frac{|x|-n\ell}{\sigma\sqrt{n}},$$
and therefore it holds 
\begin{align*}
    M_n^*((-\infty,x]) = M_n(A(0, n\ell+x\sqrt{n}\sigma)) \quad \text{for } x\in\R.
\end{align*}
We use the characterization of convergence in probability via subsequences \cite[Corollary 6.13, Page 151]{Klenke}, and for this let $(a_n)_{n\in \N}$ be an increasing sequence of natural numbers. In view of Lemma \ref{lem: CLT helping lemma 3}, there exists a subsequence $(a_{n_k})_{k\in \N}$ such that 
\begin{align*}
\lim_{k\to\infty}\int_{0}^y \e^{-it\frac{\sqrt{a_{n_k}}\ell}{\sigma}}\Psi_{a_{n_k}}\Big( \frac{t}{\sigma\sqrt{a_{n_k}}}\Big) \,dt= W\cdot
         \int_0^y   \varphi(t) \,dt, \quad \text{almost surely}
\end{align*}
for every $y\in\R$, which is equivalent to 
\begin{align}\label{eq:as-subseq}
 \lim_{k\to\infty}   \int_{-\infty}^\infty f \,dM_{a_{n_k}}^* = W\cdot\frac{1}{\sqrt{2\pi}}\int_{-\infty}^\infty f(t)\e^{-\frac{t^2}{2}}\,dt \quad \text{almost surely}
\end{align}
for every $f\in\mathcal{S}$, where $\mathcal{S}$ is defined as
\begin{align*}
    \mathcal{S}:=\Big\lbrace x\mapsto\int_0^y\e^{itx}\,dt\Big|y\in\R\Big\rbrace.
\end{align*}
Indeed, Equation \eqref{eq:as-subseq} holds in view of the following computation for every $n\in\N$:
\begin{align*}
\int \bigg( \int_0^y \e^{itx}\,dt\bigg) \,dM_n^*(x) &= \int_0^y \int \e^{itx}\,dM_n^*(x) \,dt = \int_0^y \frac{1}{\rho^n}\sum_{v\in T_n} \e^{it\frac{|X_v|-n\ell}{\sigma\sqrt{n}}}\,dt \\
&= \int_0^y \e^{-it\frac{\sqrt{n}\ell}{\sigma}} \frac{1}{\rho^n}\sum_{v\in T_n}e^{i\frac{t}{\sigma\sqrt{n}}|X_v|}\,dt
= \int_0^y \e^{-it\frac{\sqrt{n}\ell}{\sigma}}\Psi_n\Big(\frac{t}{\sigma\sqrt{n}}\Big) \,dt.
\end{align*}
Lemma \ref{lem: CLT helping lemma 4} implies that the family of measures $( M^*_{a_{n_k}})_{k\in\N}$ is tight almost surely. Moreover it holds that
\begin{align*}
 \sup_{n\in\R}M_n^*(\R) = \sup_{n\in\R} W_n < \infty
\end{align*}
almost surely. Since $\mathcal{S}$ is a separating family of functions by Lemma \ref{lem : separating}, we can thus use Theorem \ref{thm: Levy} to obtain that 
\begin{align*}
\lim_{k\to\infty} M^*_{a_{n_k}} = W\cdot\mathcal{N}(0,1), \quad \text{weakly almost surely}.
\end{align*}
Now let $x\in\R$ be a real number. Since weak convergence is equivalent with convergence in distribution, the Portmanteau theorem implies that the limit
    \begin{align*}
        \lim_{k\to\infty} M^*_{a_{n_k}}((-\infty,x]) = W\cdot\Phi(x)
    \end{align*}
holds almost surely.
By using again the characterization, we finally obtain that  
\begin{align*} M_n (A (0, n\ell +x\sqrt{n}\sigma)) = M_n^* ((-\infty,x]) \CIP W\cdot\Phi(x),
    \end{align*}
which proves the statement of Theorem \ref{thm:second_main}.
\end{proof}

\section{Branching anisotropic random walks on homogeneous trees}\label{sec: example}
In this final section we apply Theorem \ref{thm:first_main} and Theorem \ref{thm:second_main} to the concrete example of anisotropic random walks on homogeneous trees.

\textbf{Anisotropic random walks on homogeneous trees.}
Consider a countable infinite group $G$ with symmetric generating set $S\subset G$. The Cayley graph $\Gamma(G,S)$ of $G$ has all elements of $G$ as vertices, and $a,b\in G$ are connected by an edge if and only if $ab^{-1}\in S$. Note that the condition that $S$ is symmetric is necessary for this definition to be well defined, and the condition that $S$ is generating is necessary in order for $\Gamma(G,S)$ to be connected. A Cayley graph is always vertex transitive and thus regular i.e.~each vertex has the same degree, which is finite if and only if $S$ is finite.
The homogeneous tree $\mathbb{T}_d$ of degree $d\geq 3$ is a regular tree where every vertex has  $d$ neighbours, and it can be realized as a Cayley graph of the $d$-fold free product of the group $\Z_2$, i.e.
$G = \Z_2*\Z_2*\cdots*\Z_2$.
In order to define an anisotropic random walk on $\mathbb{T}_d$ we define random walks on groups first. Let $\mu$ be a probability measure on $G$ and consider a sequence of i.i.d.~random variables $(\xi_n)_{n\in\N}$ distributed according to $\mu$. The random walk started in $x\in G$ is then given by the sequence $(Y_n^x)_{n\in\N}$, where
$$Y_n^x = x\xi_1\xi_2\cdots \xi_n.$$
It the walk starts at the identity of $G$, we write $(Y_n)_{n\in N}$. For $a,b\in G$,
the transition probabilities can be written as
$$\Prob(Y_{n+1} = b|Y_n = a) = \mu(a^{-1}b),$$
and thus a random walk on $G$ is a homogeneous Markov chain on $G$, and therefore on $\Gamma(G,S)$ ($S$ is some symmetric and finite generating set for $G$). We then call a random walk on $\mathbb{T}_d$ \textit{isotropic} if the probabilities $\mu(g), g\in S$ are identical (an example is the simple random walk); otherwise the walk is called \textit{anisotropic}.
It is well known that the speed of the simple random walk on $\mathbb{T}_d$ is $1-2/d$, and that the random walk satisfies a central limit theorem. This is due to the fact that $d(o,Y_n)=|Y_n|$ is a reflected random walk on $\Z$ with drift $1\cdot (d-1)/d + (-1)\cdot 1/d = 1-2/d$ . Sawyer-Steger \cite{SawSteg} have proven a law of large numbers and a central limit theorem for the rate of escape of anisotropic random walks on homogeneous trees. The law of large numbers follows from Kingman's subadditive ergodic theorem \cite{SubET}, but in \cite{SawSteg} the authors considered 
more general quantities than $|Y_n|$. 
Moreover, they proved a central limit theorem and a law of iterated logarithm for these quantities and therefore for $|Y_n|$. The central limit theorem states that if $\E[|Y_1|^{4+\delta}]< \infty$ for some $\delta > 0$ and if the support of the law of $Y_1$ is not contained in a proper subgroup of $G$, then there exists a constant $\sigma>0$, such that 
$$\frac{|Y_n|-n\ell}{\sigma\sqrt{n}} \CID \mathcal{N}(0,1), \quad \text{as }n\to\infty$$
and
$$\sigma = \limsup_{n\to\infty} \frac{|Y_n|-n\ell}{\sqrt{2n\log\log n}} \quad \text{almost surely.}$$

\textbf{Branching anisotropic random walks on $\mathbb{T}_d$.}
We use the terminology \textit{branching anisotropic random walk} for branching random walks in which the underlying Markov chain is an anisotropic random walk. 
From the exponential moment condition $\E[\e^{t|Y_1|}] <\infty$ for some $t>0$, it follows that $\E[|Y_1|^{4+\delta}]<\infty$ for some $\delta > 0$ (in fact for every $\delta>0$), so assumptions \ref{A3} and \ref{A4} from Theorem \ref{thm:first_main} and Theorem \ref{thm:second_main} are fullfiled. Thus together with the results from \cite{SawSteg}, Theorem \ref{thm:first_main} implies that for branching anisotropic random walks on homogeneous trees $\mathbb{T}_d$, the average distance of particles from the origin on $\mathbb{T}_d$ grows almost surely linearily with growth rate $\ell>0$: 
\begin{align*}
 \lim_{n\to\infty}\frac{1}{n\rho^n}\sum_{|v|=n} |X_v| = \ell\cdot W \quad \text{almost surely}.
\end{align*}
Moreover Theorem \ref{thm:second_main} implies that a Stam-type central limit theorem holds, i.e.~for any two fixed real numbers $a<b$:
$$\rho^{-n}\big|\lbrace  v\in T_n: |X_v|\in [n\ell+a\sqrt{n}\sigma,n\ell+ b\sqrt{n}\sigma ] \rbrace\big| \CIP W\cdot\big(\Phi(b)-\Phi(a)\big), \quad \text{as } n\to\infty.$$

\textbf{Simulations.}
We simulate the results on $\mathbb{T}_3$, the homogeneous tree of degree $3$, i.e.~the Cayley graph of $\Z_2*\Z_2*\Z_2$. Every vertex can be represented as a finite word over the alphabet $\{a,b,c\}$, such that two consecutive letters cancel. We choose the one-step distribution 
$$\mu(a) = 0.1,\, \mu(b) = 0.2, \, \mu(c) = 0.1,\, \mu(ab) = 0.15,\, \mu(abc) = 0.15,\, \mu(ac) = 0.3$$
as well as the offspring distribution for the branching process.
$$\pi = \frac{1}{2}\delta_1 + \frac{1}{2}\delta_2.$$
\begin{center}
\begin{figure}[h]
  \centering
  \begin{minipage}[b]{0.49\textwidth}
    \includegraphics[width=\textwidth]{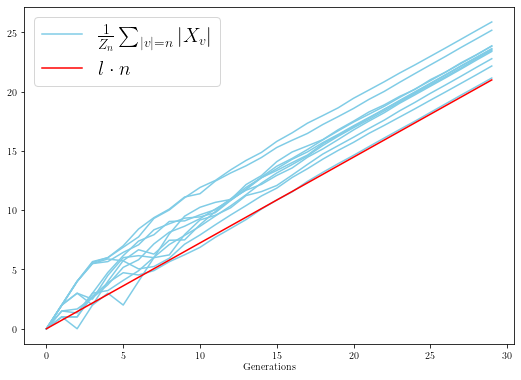}
    \caption{Normalized by $Z_n$}\label{fig:sim1}
  \end{minipage}
  \hfill
  \begin{minipage}[b]{0.49\textwidth}
    \includegraphics[width=\textwidth]{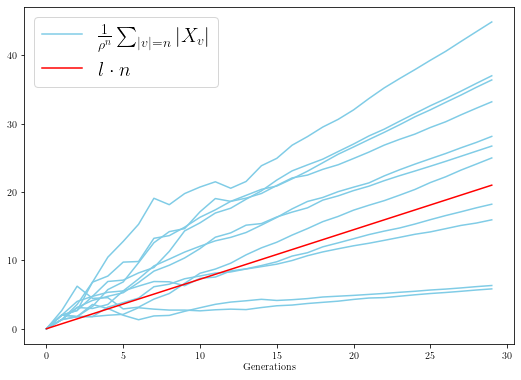}
    \caption{Normalized by $\rho^n$}\label{fig:sim2}
  \end{minipage}
\end{figure}
\end{center}
\vspace{-0.5cm}
\begin{center}
\begin{figure}[h]
  \begin{minipage}[b]{0.49\textwidth}
    \includegraphics[width=\textwidth]{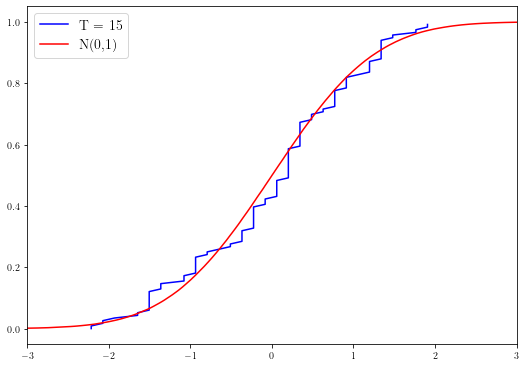}
    \caption{$T = 15$}\label{fig:sim3}
  \end{minipage}
  \hfill
  \begin{minipage}[b]{0.49\textwidth}
    \includegraphics[width=\textwidth]{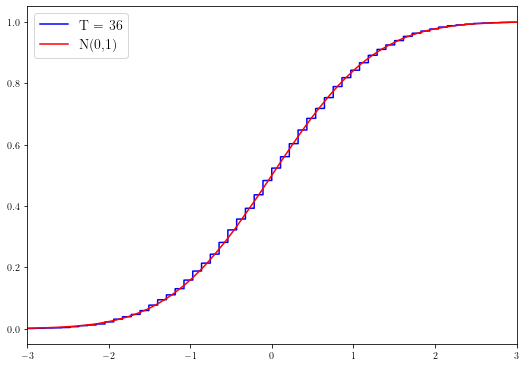}
    \caption{$T = 36$}\label{fig:sim4}
  \end{minipage}
\end{figure}
\end{center}
We have simulated the average displacement (in Figure \ref{fig:sim1} divided by the actual number of particles $Z_n$, and in Figure  \ref{fig:sim2} divided by $\rho^n$) and the rescaled empirical distribution at time $T = 15$ and time $T=36$ in Figure \ref{fig:sim3} and Figure \ref{fig:sim4} respectively.
For the law of large numbers one can see how the average displacement grows linearly with growth rate $\ell$ times $W$, and one can also see that the empirical distribution is already after a few steps close to a normal distribution.

\section*{Appendix}
We collect below several auxiliary results needed in Section \ref{sec:clt}.
\begin{lemma}\label{lem: L2 population martingale}
If $\theta<\infty$, then the population martingale $W_n = Z_n/\rho^n$ converges in $L^2$ to the limit $W$.
\end{lemma}
\begin{proof}
Recall the definition of a Galton-Watson process in Section \ref{sec:prelim}. By the calculations from the proof of Lemma \ref{lemma: decay of splitting time}, we have that
 $$\mathbb{E}\big[W_{n+1}^2\big]=\mathbb{E}\big[W_n^2\big]+\frac{1}{\rho^{n}}\Big(\frac{\theta-\rho^2}{\rho^2}\Big).$$
This recursive equation can be solved and its solution is given by the closed-term expression
 $$\mathbb{E}\big[W_n^2\big]=1+\Big(\frac{\theta-\rho^2}{\rho^2}\Big)\sum_{k=0}^{n-1}\frac{1}{\rho^{k}},$$
thus $\sup_{n\in\N} \E\big[W_n^2\big] < \infty$,
which together with \cite[Theorem 11.10]{Klenke}  yields the result.
\end{proof}
Recall that by $\varphi$ we have denoted the characteristic function of a standard Gaussian random variable.
\begin{lemma}\label{lem: CLT helping lemma 2}
Let $(R_n)_{n\in\N}$  be a sequence random variables with characteristic functions $\varphi_{R_n}$, such that for all $t\in\R$
\begin{align*}
\lim_{n\to\infty}\varphi_{R_n}(t) = \varphi(t),
\end{align*}
 and let $(t_n)_{n\in\N}$ be a sequence of real numbers with limit $t\in \R$. Then it holds
\begin{align*}
 \lim_{n\to\infty}\varphi_{R_n}(t_n) = \varphi(t).
 \end{align*}
\end{lemma}
\begin{proof}
During the proof we write $\varphi_n$ for $\varphi_{R_n}$. Using the triangle inequality, we obtain
    \begin{align*}
        | \varphi_n(t_n) - \varphi(t)|\leq 
        | \varphi_n(t_n) - \varphi_n(t)|+
        | \varphi_n(t) - \varphi(t)|.
    \end{align*}
The second term converges to $0$ by assumption. For the first term let $R$ be a standard Gaussian random variable. By setting $\varepsilon_n = t_n -t$ and using the bound $\vert e^{it}-1\vert\leq \min\{\vert t \vert ,2\}$ we obtain
    \begin{align*}
 | \varphi_n(t_n) - \varphi_n(t)| &= \big|\E[\e^{it_nR_n}-\e^{itR_n}]\big|\leq \E[|\e^{it_nR_n}-\e^{itR_n}|] \\
 &= \E[|\e^{i\varepsilon_nR_n}-1|]\leq \E[\min\lbrace |\varepsilon_nR_n|,2\rbrace].
    \end{align*}
    Now let $\varepsilon>0$. Since $t_n\rightarrow t$, we have for $n$ large enough that
    $$\mathbb{E}[\min\{\vert\varepsilon_n R_n\vert ,2\}]\leq \mathbb{E}[\min\{\vert\varepsilon R_n\vert ,2\}].$$
    Since the function
    \begin{align*}
        \mathbb{R}\to [0,\infty):t\mapsto \min \lbrace |\varepsilon t|,2 \rbrace 
    \end{align*}
    is bounded and continuous, pointwise convergence of the characteristic functions implies weak convergence, from which we obtain
    \begin{align*}
         \lim_{n\to\infty}\E[\min\lbrace |\varepsilon R_n|,2\rbrace] = \E[\min\lbrace |\varepsilon R|,2\rbrace] \leq \E[|\varepsilon R|] = \frac{2}{\sqrt{2\pi}}\varepsilon\leq \varepsilon.
    \end{align*}
    We thus see that
    $$\lim_{n\to\infty}\E[\min\lbrace |\varepsilon_n R_n|,2\rbrace]\leq \varepsilon.$$
Since the limit holds for an arbitrary $\varepsilon$, the proof is completed.
\end{proof}

\begin{lemma}\label{lem: CLT helping lemma 4}
Let $\mu$, $(\mu_n)_{n\in \N}$ be measures on $\R$ with corresponding characteristic functions $\varphi$ and $(\varphi_n)_{n\in \N}$ respectively, such that for every $y\in\R$ it holds
\begin{align*}
\lim_{n\to\infty}\int_0^y\varphi_{\mu_n}(t)\,dt = \int_0^y\varphi_\mu(t)\,dt.
    \end{align*}
Then the family of measures $(\mu_n)_{n\in\N}$ is tight.
\end{lemma}
\begin{proof}
We follow the proof of \cite[Theorem 15.24, Page 345]{Klenke}. We define the function 
$h:\R\to[0,\infty)$ by
\begin{align*}
 h(x) = 1-\frac{\sin(x)}{x},\quad \text{for }x\neq 0,
\end{align*}
and $h(0)=0$, which is clearly continuously differentiable on $\R$. We set $\alpha = \inf\lbrace h(x):|x|\geq 1\rbrace = 1-\sin(1) > 0$. For $K>0$, in view of Markov's inequality together with Fubini's theorem we obtain
    \begin{align*}
        \mu_n([-K,K]^c)&\leq \frac{1}{\alpha}\int_{[-K,K]^c} h(x/K)\,d\mu_n(x)\\
        &\leq \frac{1}{\alpha}\int_\R h(x/K)\,d\mu_n(x)\\
        & = \frac{1}{\alpha}\int_\R\Big(\int_0^1 (1-\cos(tx/K))\,dt\Big) \,d\mu_n(x) \\
        & = \frac{1}{\alpha}\int_0^1 \Big(\int_\R (1-\cos(tx/K))\,d\mu_n(x)\Big) \,dt  \\
        & =\frac{1}{\alpha}\int_0^1 \big(1-\text{Re}(\varphi_{\mu_n}(t/K))\big)\,dt \\
        &=\frac{1}{\alpha}\Big(1-\text{Re}\big(\int_0^1\varphi_{\mu_n}(t/K)\big)\Big) .
    \end{align*}
Now we can use the assumption and the dominated convergence to obtain that 
    \begin{align*}
        \limsup_{n\to\infty} \mu_n([-K,K]^c) &\leq \limsup_{n\to\infty}\frac{1}{\alpha}\int_0^1 \big(1-\text{Re}(\varphi_{\mu_n}(t/K))\big)\,dt \\
        &=\frac{1}{\alpha}\Big(1-\text{Re}\big(\lim_{n\to\infty}\int_0^1\varphi_{\mu_n}(t/K)\big)\Big) \\
        &= \frac{1}{\alpha}\Big(1-\text{Re}\big(\int_0^1\varphi_{\mu}(t/K)\big)\Big) \\
        &= \frac{1}{\alpha}\int_0^1 \big(1-\text{Re}(\varphi_\mu(t/K))\big)\,dt.
    \end{align*}
Since $\varphi_\mu$ is a characteristic function, it is continuous. Therefore if $K\rightarrow\infty$ the last integral and thus the expression on the left-hand side in the previous equation converges to $0$ and this proves the tightness of the sequence $(\mu_n)_{n\in \N}$.
\end{proof}

\begin{lemma}\label{lem : separating}
The family $\mathcal{S}$ of functions
    \begin{align*}
        \mathcal{S} = \left\lbrace x \mapsto \int_0^y e^{itx} \,dt\Big|y\in\R\right\rbrace \subset \mathcal{C}_b(\R)
    \end{align*}
    is separating.
\end{lemma}
\begin{proof}
Let $\mu\neq\nu$ be two measures on $\R$. Assume that for every $y\in\R$, we have
\begin{align*}
\int\int_0^y\e^{itx} \, dt\,d\mu(x) =  \int\int_0^y\e^{itx} \, dt\,d\nu(x).
\end{align*}
Then it holds that
\begin{align*}
\int_0^y \varphi_\mu(t)\,dt &= \int_0^y \int \e^{itx} \,d\mu(x) \,dt =   \int\int_0^y\e^{itx} \, dt\,d\mu(x)  \\
        & = \int\int_0^y\e^{itx} \, dt\,d\nu(x) =  \int_0^y \int \e^{itx} \,d\nu(x) \,dt=    \int_0^y \varphi_\nu(t)\,dt
\end{align*}
for every $y\in\R$, which implies that 
    \begin{align*}
        \varphi_\mu(y) = \frac{d}{dy}\int_0^y \varphi_\mu(t)\,dt = \frac{d}{dy}\int_0^y \varphi_\nu(t)\,dt =   \varphi_\nu(y) 
    \end{align*}
    for every $y\in\R$. But since a measure is uniquely determined by its characteristic function, this means that $\mu = \nu$  which is a contradiction.
\end{proof}

\section*{Concluding remarks and questions}

\textbf{Maximal displacement.}
In Section \ref{subsec: Maximal Displacement} we have proven that the maximal displacement 
$$\maxdis_n = \max_{v\in T_n}|X_v|,$$
grows at most linearly, but we did not prove that the growth rate is linear, i.e.~there exists an $\alpha>0$, such that
$$\lim_{n\to\infty}\frac{\maxdis_n}{n} = \alpha,\quad \text{almost surely}.$$
For supercritical branching random walks on $\Z$ and $\R$, this was shown in a series of works \cite{FirstBirthProbBigg,FirstBirthProbKing,FirstBirthProbHamm}. In \cite{HuandShi}, the authors improved the result by establishing the existence of a logarithmic correction term, i.e.~they have shown that if $\alpha$ is the linear speed of the maximal displacement it holds
$$\frac{\maxdis_n-\alpha n}{\log(n)} \CIP \frac{3}{2}, \quad \text{as } n\to\infty.$$
We conjecture that the maximal displacement for a supercritical branching random walk on a transitive graph exhibits the same behaviour. For branching random walks  on supercritical Galton-Watson trees, the question of the linear speed of the maximal displacement has been addressed recently  in \cite{maximal-displacement-tree}.

\textbf{Critical branching random walk.}
In the current paper, we have proven limit theorems for the empirical distribution of a supercritical branching random walk. It is interesting to consider the case where the branching process is critical, i.e.~the expected number of offspring is equal to $1$. In the critical case, the population dies out almost surely, and we can consider the critical process conditioned on survival as introduced by Kesten in \cite{KestenTree}.
A Stam-type central limit theorem for conditional critical branching random walks on $\R$ has been proven in \cite{CriticalStam}. See Kesten \cite{CriticalDisplacement1} and Lalley-Shao \cite{CriticalDisplacement2} for limit results on the maximal displacement of critical branching random walks.

In our context, it would be interesting to investigate the empirical distribution of conditional critical branching random walks, and to prove a law of large numbers for the cumulative rate of escape and a Stam-type central limit theorem on transitive graphs. Central limit theorems for critical branching random walks conditioned on survival have been studied on $\R$ in \cite{CriticalStam}.  In the critical branching random walks conditioned to survive on transitive graphs, we expect to obtain completely different limit results than in the supercritical case.

\textbf{Funding information.} The research of R. Kaiser and E. Sava-Huss was funded in part by the Austrian Science Fund (FWF) 10.55776/PPAT3123425. For open access purposes, the authors have applied a CC BY public copyright license to any author-accepted manuscript version arising from this submission.

\bibliography{Literature}
\bibliographystyle{alpha}
\end{document}